\documentclass[final,leqno,onefignum,onetabnum]{siamltex1213}
\usepackage{amssymb,amsmath,bm}
\usepackage{color}
\usepackage{comment}
\usepackage{calrsfs}
\usepackage{amsfonts}
\usepackage{mathrsfs}
\usepackage{rotating}
\usepackage{multirow}
\DeclareMathAlphabet{\pcal}{OMS}{zplm}{m}{n}

\newtheorem{remark}{Remark}

\newcommand{\gf}{f^{\diamond}}
\newcommand{\gig}{g^\diamond}
\newcommand{\gh}{h^\diamond}

\newcommand{\R}{{\mathbb{R}}}
\newcommand{\C}{{\mathbb{C}}}
\newcommand{\Rmn}{{\R^{m\times n}}}
\newcommand{\Cmn}{{\C^{m\times n}}}
\newcommand{\Cnn}{\C^{n\times n}}
\newcommand{\wt}{\widetilde}

\newcommand{\tref}[1]{Table~\ref{#1}}

\title{COMPUTATION OF GENERALIZED MATRIX FUNCTIONS}
\author{Francesca Arrigo,\footnotemark[2]\,  Michele Benzi,\footnotemark[3]\,  
 and Caterina Fenu\footnotemark[4]}
\begin{document}
\maketitle

\renewcommand{\thefootnote}{\fnsymbol{footnote}}
\footnotetext[2]{Department of Science and High Technology, 
University of Insubria, Como 22100, Italy (\email 
francesca.arrigo@uninsubria.it).}
\footnotetext[3]{Department of Mathematics and Computer science, 
Emory University, Atlanta, Georgia 30322, USA (\email benzi@mathcs.emory.edu).
The work of this author was supported by National Science Foundation grant
DMS-1418889.} 
\footnotetext[4]{Department of Computer Science, 
University of Pisa, Pisa 56127, Italy (\email caterina.fenu@for.unipi.it). 
The work of this author was supported by INdAM-GNCS Grant 2014/001194.}
\renewcommand{\thefootnote}{\arabic{footnote}}

\begin{abstract}
We develop numerical algorithms for the efficient evaluation of quantities associated
with generalized matrix functions [J.~B.~Hawkins and A.~Ben-Israel, {\em Linear and
Multilinear Algebra}, 1(2), 1973, pp.~163--171]. Our algorithms are based on 
Gaussian quadrature and Golub--Kahan bidiagonalization. Block variants are also
investigated. Numerical experiments are performed to illustrate the effectiveness and
efficiency of our techniques in computing generalized matrix functions arising
in the analysis of networks.
\end{abstract}



\pagestyle{myheadings}
\thispagestyle{plain}
\markboth{{\sc F.~Arrigo, M.~Benzi, and C.~Fenu}}{Generalized matrix functions}

\section{Introduction}

Generalized matrix functions were first introduced by Hawkins and Ben-Israel 
in \cite{HBI73} in order to extend the notion of a matrix function to
rectangular matrices. Essentially, the definition is based on replacing
the spectral decomposition of $A$ (or the Jordan canonical form, if $A$ is not
diagonalizable) with the singular value decomposition, and evaluating the
function at the singular values of $A$, if defined. 
While it is likely that this definition was inspired
by the analogy between the inverse and the Moore--Penrose generalized inverse,
which is well established in numerical linear algebra,
\cite{HBI73} is a purely theoretical paper and does not mention any potential 
applications or computational aspects. The paper appears to have gone 
largely unnoticed, despite 
increasing interest in matrix functions in the numerical linear algebra
community over the past several years; for instance, it is not cited 
in the important monograph by Higham 
\cite{BookH}. While it is likely that the perceived scarcity of applications
is to blame (at least in part) for this lack of attention, it
turns out that generalized matrix functions do have interesting
applications and have actually occurred
in the literature without being recognized as such; see section
\ref{sec:app} for some examples. 

In this paper we revisit the topic of generalized matrix functions, with
an emphasis on numerical aspects. After reviewing the necessary background
and definitions, we consider a few  situations naturally leading to 
generalized matrix functions. Moving on to numerical considerations, we develop several 
computational approaches based on variants of Golub--Kahan bidiagonalization 
to compute or estimate bilinear forms involving generalized matrix functions,
including entries of the generalized matrix function itself and the action
of a generalized matrix function on a vector. We further consider block variants
of Golub--Kahan bidiagonalization which can be used to evaluate matrix-valued 
expressions involving generalized matrix functions.   
Numerical experiments are used to illustrate the performance
of the proposed techniques on problems arising in the analysis of
directed networks. 

\section{Background}\label{sec:background}
In this section we review a few basic concepts from linear algebra that 
will be used throughout the paper, mostly to set our notation,
and recall the notion of generalized
matrix function.  

Let $A\in\Cmn$ and let $r$ be the rank of $A$. 
We can factor the matrix $A$ as $A = U\Sigma V^*$ using a singular value decomposition 
(SVD). 
The matrix $\Sigma\in\Rmn$ is 
diagonal and its entries $\Sigma_{ii} = \sigma_i$ are ordered as  
$\sigma_{1}\geq\sigma_{2}\geq\cdots\geq\sigma_{r}>\sigma_{r+1}=\cdots=\sigma_{q}=0$, 
where $q=\min\{m,n\}$. The $r$ positive $\sigma_i$ are the singular values of $A$. 
The matrices  
$U = [{\bf u}_1,{\bf u}_2,\ldots,{\bf u}_m]\in\C^{m\times m}$ and 
$V = [{\bf v}_1,{\bf v}_2,\ldots,{\bf v}_n]\in\C^{n\times n}$ are unitary and contain the left 
and right singular vectors of 
$A$, respectively. 
It is well known that the matrix $\Sigma$ is uniquely determined, while $U$ and $V$ are not. 
If $A$ is real, then $U$ and $V$ can be chosen to be real.
From the singular value decomposition of a matrix $A$ it follows that 
$AA^* = U\Sigma\Sigma^T U^*$ and $A^*A = V \Sigma^T\Sigma V^*$. 
Thus, the singular values of a matrix $A$ are the square roots of the 
positive eigenvalues of the matrix $AA^*$ or $A^*A$. 
Moreover, the left singular vectors of $A$ are the eigenvectors of the matrix 
$AA^*$, while the right singular vectors are the 
eigenvectors of the matrix $A^*A$. 
The singular values of a matrix also arise (together with their opposites) as the 
eigenvalues of the Hermitian matrix 
\begin{equation}\label{eq:calA}
\mathscr{A} =\left(
\begin{array}{cc}
0 & A\\
A^* & 0
\end{array}\right).
\end{equation}   
This can be easily seen for the case $m=n$; indeed, under this hypothesis, the 
spectral factorization of $\mathscr{A}$ is given by~\cite{BookHJ}:
\begin{equation*}
\mathscr{A} = \frac{1}{2}
\left(
\begin{array}{cc}
U & -U \\
V & V
\end{array}
\right)
\left(
\begin{array}{cc}
\Sigma & 0 \\
0 & -\Sigma
\end{array}
\right)
\left(
\begin{array}{cc}
U & -U \\
V & V
\end{array}
\right)^* .
\end{equation*}

Consider now the matrices $U_r\in\C^{m\times r}$ and $V_r\in\C^{n\times r}$ which contain the 
first $r$ columns of the matrices $U$ and $V$, respectively, and let 
$\Sigma_r\in\R^{r\times r}$ be the $r\times r$ leading block of $\Sigma$. 
Then a {\it compact SVD} ({\it CSVD}) of the matrix $A$ is 
$$A = U_r\Sigma_r V^*_r = \sum_{i=1}^r\sigma_i{\bf u}_i{\bf v}_i^*.$$

\subsection{Matrix functions}
There are several equivalent ways to define $f(A)$ when $A\in\Cnn$ is a square matrix. 
We recall here the definition based on the Jordan canonical form. For a 
comprehensive study of matrix functions, we refer to~\cite{BookH}. 
 
Let $\{\lambda_1,\lambda_2,\ldots,\lambda_s\}$ be the set of distinct eigenvalues of $A$ and let 
$n_i$ denote the {\em index} of the $i$th eigenvalue, i.e., the size of the largest Jordan
block associated with $\lambda_i$. 
Recall that a function $f$ is said to be {\it defined on the spectrum of $A$} 
if the values $f^{(j)}(\lambda_i)$ exist for all 
$j=0,\ldots,n_i-1$ and for all $i=1,\ldots,s$, where $f^{(j)}$ is the $j$th derivative 
of the function and $f^{(0)}=f$.

\begin{definition}{\rm \cite[Definition 1.2]{BookH}}\label{def:fun}
Let $f$ be defined on the spectrum of $A\in\C^{n\times n}$ and let $Z^{-1} A Z = J = \diag(J_1,J_2,\ldots,J_p)$ be the 
Jordan canonical form of the matrix, where 
$$J_k = J_k(\lambda_k) =
\left(
\begin{array}{cccc}
\lambda_k &    1      &        & \\
          & \lambda_k & \ddots & \\
          &           & \ddots & 1 \\
          &           &        & \lambda_k  
\end{array}
\right)\in\C^{m_k\times m_k},$$ 
$\sum_{k=1}^pm_k=n$, and $Z$ is nonsingular. 
Then 
$$f(A) := Zf(J)Z^{-1} = Z \diag(f(J_1),f(J_2),\ldots,f(J_p)) Z^{-1},$$
where 
$$f(J_k) :=
\left(
\begin{array}{cccc}
f(\lambda_k) & f'(\lambda_k)  & \cdots & \frac{f^{(m_k-1)}(\lambda_k)}{(m_k-1)!}\\
             & f(\lambda_k)        & \ddots & \vdots \\
             &                     & \ddots & f'(\lambda_k) \\
             &                     &        & f(\lambda_k)   
\end{array}
\right),$$
 
\end{definition}

If the matrix $A$ is diagonalizable, then the Jordan canonical form reduces to the spectral 
decomposition: $A = ZDZ^{-1}$ 
with $D=\diag(\lambda_1,\lambda_2,\ldots,\lambda_n)$. 
In such case, $f(A) = Z\diag(f(\lambda_1),f(\lambda_2),\ldots,f(\lambda_n))Z^{-1}$. 

If the function $f$ has a Taylor series expansion, we can use it to describe the associated matrix function, provided that the eigenvalues of the matrix $A$ satisfy certain requirements.
\begin{theorem}{\rm \cite[Theorem 4.7]{BookH}}
Let $A\in\Cnn$ and suppose that $f$ can be expressed as 
$$f(z) = \sum_{k=0}^\infty a_k(z-z_0)^k$$
with radius of convergence $R$. 
Then $f(A)$ is defined and is given by $$f(A) = \sum_{k=0}^\infty a_k (A-z_0I)^k$$
if and only if each of the distinct eigenvalues of $A$ 
$\{\lambda_1,\lambda_2,\ldots,\lambda_s\}$ satisfies one of the following:
\begin{itemize}
\item[(i)] $|\lambda_i - z_0|< R$; 
\item[(ii)] $|\lambda_i - z_0| = R$ and the series for $f^{(n_i - 1)}(\lambda)$ is 
convergent at the point $\lambda = \lambda_i$, $i=1,2,\ldots, s$. 
\end{itemize}
\end{theorem}

\subsection{Generalized matrix functions}
In \cite{HBI73} the authors considered the problem of defining functions 
of rectangular matrices. 
Their definition relies on the following generalization of the SVD.
\begin{theorem}\label{thm:ci}
	Let $A\in\C^{m\times n}$ be a matrix of rank $r$ and let 
	$\{c_i: i=1,2,\ldots, r\}$ be any complex numbers satisfying 
	$$|c_i|^2 = \sigma_i^2 = \lambda_i(AA^*),$$
	where $\lambda_1(AA^*)\geq\lambda_2(AA^*)\geq\cdots\geq\lambda_r(AA^*)>0$ are the positive eigenvalues of $AA^*$. 
	Then there exist two unitary matrices ${\pcal X}\in\C^{m\times m}$ 
	and ${\pcal Y}\in\C^{n\times n}$ such 
	that ${\pcal D} = {\pcal X}^*A{\pcal Y} \in\Cmn$ has entries: 
	$$d_{ij}=\left\{
	\begin{array}{ll}
	c_i & {\rm if  }\ 1 \leq i=j\leq r\,, \\
	0   & {\rm otherwise.}
	\end{array}
	\right .$$ 
\end{theorem}
From this theorem it follows that, once the non-zero entries of 
$\pcal{D}$ are fixed,  $A$ can be written as 
\begin{equation}\label{eq:factorization}
A={\pcal X}{\pcal D}{\pcal Y}^*={\pcal X}_r{\pcal D}_r{\pcal Y}_r^*,
\end{equation} 
where 
${\pcal D}_r$ is the leading $r\times r$ block of ${\pcal D}$ and the matrices ${\pcal X}_r\in\C^{m\times r}$ 
and ${\pcal Y}_r\in\C^{n\times r}$ consist of the first $r$ columns of the matrices 
${\pcal X}$ and ${\pcal Y}$, respectively.  

In this paper we do not make use of the extra degrees of freedom provided 
from this (slightly) more general SVD of $A$, and we assume that 
$c_i=\sigma_i$ for all $i=1,2,\ldots,r$. This assumption 
ensures that the decompositions in \eqref{eq:factorization} 
coincide with the SVD and CSVD of the matrix $A$, respectively. 
In particular, ${\pcal D} = \Sigma$, ${\pcal X} = U$, and ${\pcal Y} = V$.
All the definitions and results presented in the remaining of this section 
and in the next one can be 
extended to the case when the coefficients $\{c_i: i=1,2,\ldots,r\}$ 
do not necessarily coincide with the singular values, but satisfy the 
hypothesis of Theorem~\ref{thm:ci}. 

\begin{definition}\label{def:gf}
Let $A\in\Cmn$ be a rank $r$ matrix and let $A = U_r\Sigma_rV_r^*$ be its CSVD.
Let $f:\R\rightarrow\R$ be a scalar function such that $f(\sigma_i)$ is 
defined for all $i=1,2,\ldots,r$. 
The {\rm generalized matrix function} $\gf:\Cmn\rightarrow \Cmn$ 
{\rm induced by} $f$ is defined as 
$$\gf(A) := U_rf(\Sigma_r)V_r^*,$$
where $f(\Sigma_r)$ is defined for the 
square matrix $\Sigma_r$ according to definition~\ref{def:fun} as
$$f(\Sigma_r)=\diag(f(\sigma_1),f(\sigma_2),\ldots,f(\sigma_r)).$$
\end{definition}

As already mentioned in the Introduction,
generalized matrix functions arise, for instance,
when computing $f(\mathscr{A})$, where $\mathscr{A}$ is the 
matrix defined in~\eqref{eq:calA}. 
Indeed, if one uses the description of matrix function in terms of power series 
$f(z) = \sum_{k=0}^\infty a_k z^k$, it is easy to check that, within the radius of convergence:
$$f(\mathscr{A}) =\left( 
\begin{array}{cc}
f_{even}(\sqrt{AA^*}) & \gf_{odd}(A) \\
\gf_{odd}(A^*)   & f_{even}(\sqrt{A^*A}) 
\end{array}
\right),$$
where $$f(z) = f_{even}(z) + f_{odd}(z) = \sum_{k=0}^\infty a_{2k}z^{2k} 
+ \sum_{k=0}^\infty a_{2k+1}z^{2k+1}.$$

\begin{remark}
{\rm 
If $f(0)=0$ and the matrix $A\in\Cnn$ is Hermitian positive semidefinite,  
then  the generalized matrix 
function $\gf(A)$ reduces to the standard matrix function 
$f(A)$. If the more general decomposition of Theorem \ref{thm:ci}
is used instead, then the generalized matrix function reduces to $f(A)$ if $f(0)=0$ and the 
matrix $A\in\Cnn$ is normal, as long as $f$ is defined on the $c_i$; see \cite{HBI73}.
}
\end{remark}

\begin{remark}
{\rm
The Moore--Penrose pseudo-inverse of a matrix $A\in\Cmn$, denoted by $A^\dagger$, 
can be expressed as $\gf(A^*)$, where $f(z) = z^{-1}$.
Equivalently, $\gf(A) = (A^\dagger)^*$ when $f(z) = z^{-1}$.
Hence, there is a slight disconnect between the definition of generalized matrix function
and that of generalized inverse. This small issue could be addressed
by defining a generalized matrix function corresponding to the scalar function $f$ as
$\gf(A) = V_r f(\Sigma_r) U_r^*$, so that the generalized matrix function of an $m\times n$
matrix is an $n\times m$ matrix, and  $\gf(A) = A^\dagger$ when $f(z) = z^{-1}$. 
However, doing so would lead to the undesirable
property that $\gf(A) = A^*$ for $f(z) = z$, as well as other problems.
}
\end{remark}

\section{Properties}

In this section we review some properties of generalized matrix functions 
and we summarize a few new results. 

Letting $E_i={\bf u}_i{\bf v}_i^*$ and $E=\sum_{i=1}^rE_i$, 
we can write 
\begin{equation*}
A= \sum_{i=1}^r \sigma_i{\bf u}_i{\bf v}_i^* = \sum_{i=1}^r \sigma_iE_i,
\end{equation*}
and thus it follows that 
\begin{equation}\label{eq:AEi}
\gf(A) = \sum_{i=1}^rf(\sigma_i){\bf u}_i{\bf v}_i^* = \sum_{i=1}^r f(\sigma_i)E_i.
\end{equation}

\begin{proposition}{\rm (Sums and products of functions \cite{HBI73}).}
	Let $f,g,h:\R\rightarrow\R$ be scalar functions and let 
$\gf,\gig,\gh:\Cmn\rightarrow\Cmn$ be the 
	corresponding generalized matrix functions. Then:
	\begin{itemize}
	\item[(i)] if $f(k)=k$, then $\gf(A)=kE$;
	\item[(ii)] if $f(z)=z$, then $\gf(A)=A$;
	\item[(iii)] if $f(z)=g(z)+h(z)$, then $\gf(A)= \gig(A) + \gh(A)$;
	\item[(iv)] if $f(z)=g(z)h(z)$, then $\gf(A) = \gig(A)E^* \gh(A)$.
	\end{itemize}
\end{proposition}

\vspace{0.1in}

In the following we prove a few properties of generalized matrix functions.

\vspace{0.1in}

\begin{proposition}\label{thm:prop}
	Let $A\in\Cmn$ be a matrix of rank $r$. 
	Let $f:\R\rightarrow\R$ be a scalar function and let $\gf:\Cmn\rightarrow\Cmn$ 
        be the induced generalized matrix function, assumed to be defined at $A$. 
	Then the following properties hold true. 
	\begin{itemize}
	\item[(i)] $\left[\gf(A)\right]^* = \gf(A^*)$;
	\item[(ii)] let $X\in\C^{m\times m}$ and $Y\in\C^{n\times n}$ be two unitary matrices, then $\gf(XAY) = X[\gf(A)]Y$;
	\item[(iii)] if $A = \diag(A_{11},A_{22},\ldots,A_{kk})$, then $$\gf(A) = \diag(\gf(A_{11}),\gf(A_{22}),\ldots,\gf(A_{kk}));$$
	\item[(iv)] $\gf(I_k\otimes A) = I_k\otimes \gf(A)$, where $I_k$ is the $k\times k$ identity matrix and $\otimes$ 
			is the Kronecker product;
	\item[(v)] $\gf(A\otimes I_k)=\gf(A)\otimes I_k$.
	\end{itemize}
\end{proposition}

\begin{proof}
	\begin{itemize}
	\item[(i)] From~\eqref{eq:factorization} it follows that $A^* = V_r \Sigma_r U_r^*$, and thus 
	$$\gf(A^*) = V_r f(\Sigma_r) U_r^* = [U_r f(\Sigma_r) V_r^*]^* = [\gf(A)]^*.$$

	\item[(ii)] 

	The result follows from the fact that unitary matrices form a group under multiplication and that the rank of a matrix does not 		change under left or right multiplication by a nonsingular matrix~\cite{BookHJ}. 
	Indeed, the matrix $B := X A Y$ has rank $r$ and thus
	\begin{align*}
	\gf(B)   &= \gf(XU\Sigma V^*Y) 
		 = (XU)_r f(\Sigma_r) [(Y^*V)_r]^* \\
		 &= X U_r f(\Sigma_r) V^*_r Y  
		 = X\gf(A)Y.
	\end{align*} 
	where $(XU)_r$ and $(V^*Y)_r$ are the matrices containing the first $r$ columns of $(XU)$ and $(V^*Y)$, respectively.
	\item[(iii)] Let $A_{ii} = U_{i,r_i}\Sigma_{i,r_i} V_{i,r_i}^*$ be 
	the CSVD of the rank-$r_i$ matrix $A_{ii}$ for $i=1,2,\ldots,k$. 
	Then 
	$A =  U_r\Sigma_rV_r^* = {\mathscr U}_r{\mathscr D}_r{\mathscr V}_r^*$, where 
	$$\Sigma_r = P^T {\mathscr D}_rP := P^T \diag\left(\Sigma_{1,r_1},\ \Sigma_{2,r_2},\ldots,\ \Sigma_{k,r_k}\right)P
	$$
	is a diagonal matrix whose diagonal entries are ordered (via the permutation matrix $P$) in non-increasing order, and 
	$$
	U_r = {\mathscr U}_r P := \diag\left(U_{1,r_1},\ U_{2,r_2},\ldots,\ U_{k,r_k}\right)P,
	$$
	$$ 
	V_r = {\mathscr V}_r P := \diag\left(V_{1,r_1},\ V_{2,r_2},\ldots,\ V_{k,r_k}\right)P.
	$$
	From the definition of generalized matrix function and from some basic properties 
of standard matrix functions 
	\cite{BookH} it follows that 
	\begin{align*}
	\gf(A)  &= U_r f(\Sigma_r) V_r^* 
			= {\mathscr U}_r f({\mathscr D}_r) {\mathscr V}_r^* \\
		&= U_r \diag\left(f(\Sigma_{1,r_1}),\ f(\Sigma_{2,r_2}),\ldots,\ f(\Sigma_{k,r_k})\right) V_r^* \\
		&= \diag\left(U_{1,r_1}f(\Sigma_{1,r_1})V_{1,r_1}^*,\ U_{2,r_2}f(\Sigma_{2,r_2})V_{2,r_2}^*, 
						\ldots,\ U_{k,r_k}f(\Sigma_{k,r_k})V_{k,r_k}^*\right) \\
		&= \diag\left(\gf(A_{11}),\gf(A_{22}),\ldots,\gf(A_{kk})\right).
	\end{align*}

	\item[(iv)] The result follows from (iii) and the fact that $I_k\otimes A = \diag(A,A,\ldots,A)$ is a $km\times kn$ diagonal block 
	matrix with $k$ copies of $A$ on the main diagonal.

	\item[(v)] It follows from (iv) and from the fact that for two general matrices $A\in\Cmn$ and 
	$B\in\C^{p\times q}$, there exist two permutation matrices $K^{(p,m)}$ and $K^{(n,q)}$ called {\it commutation matrices} 
	such that $K^{(p,m)}\left(A\otimes B\right)K^{(n,q)} = B\otimes A$ (see~\cite[Ch.~3]{MN88}). 
	\end{itemize}
\end{proof}

The following theorem provides a result for the composition of two functions.

\vspace{0.1in}
 
\begin{proposition}{\rm (Composite functions)}
	Let $A\in\Cmn$ be a rank-$r$ matrix and let $\{\sigma_i:1\leq i\leq r\}$ be its singular values. 
	Assume that $h:\R\rightarrow\R$ and $g:\R\rightarrow\R$ are two scalar functions such that $h(\sigma_i)\neq 0$
	and $g(h(\sigma_i))$ exist for all 
	$i=1,2,\ldots,r$. Let $\gig:\Cmn\rightarrow\Cmn$ and $\gh:\Cmn\rightarrow\Cmn$ be the induced generalized matrix functions. 
	Moreover, let $f:\R\rightarrow\R$ be the composite function $f = g \circ h$. 
	Then the induced matrix function $\gf:\Cmn\rightarrow\Cmn$ satisfies $$\gf(A)=\gig(\gh(A)).$$ 
	
\end{proposition}

\begin{proof}
Let $B:=\gh(A) = U_r h(\Sigma_r) V_r^* = U_r \Theta_r V_r^*$. 
Since $h(\sigma_i)\neq 0$ for all $i=1,2,\ldots,r$, this matrix has 
rank $r$. We want to construct a CSVD of the matrix $B$.   
Let thus $P\in\R^{r\times r}$ be a permutation matrix such that 
the matrix
$\widetilde{\Theta}_r = P^T \Theta_r P = P^T h(\Sigma_r) P$ has diagonal entries ordered in non-increasing order. Then 
it follows that a CSVD of the matrix $B$ is given by $B = \widetilde{U}_r \widetilde{\Theta}_r \widetilde{V}_r^*$, 
where $\widetilde{U}_r = U_r P$ 
and $\widetilde{V}_r = V_r P$ have orthonormal columns. 
It thus follows that 
\begin{align*}
\gig(\gh(A)) &= \gig(B) = \widetilde{U}_rg(\widetilde{\Theta}_r)\widetilde{V}_r^* = U g(P\widetilde{\Theta}_rP^T)V \\
	    &= U_r g(\Theta_r) V_r^* = U_r g(h(\Sigma_r)) V_r^* = U_r f(\Sigma_r) V_r^* \\
	    &= \gf(A).
\end{align*}
	\qquad
\end{proof}

The following result describes the relationship between standard matrix functions and generalized matrix functions.

\begin{theorem}\label{thm:genfun}
	Let $A\in\Cmn$ be a rank-$r$ matrix and let $f:\R\rightarrow\R$ be a scalar function.
        Let $\gf:\Cmn\rightarrow\Cmn$ be the induced generalized matrix function.	 
	Then 
	\begin{subequations}
	\begin{equation}\label{eq:gen_fun1}
	\gf(A) = \left(\sum_{i=1}^r\frac{f(\sigma_i)}{\sigma_i}{\bf u}_i{\bf u}_i^*\right)A 	
	       = A\left(\sum_{i=1}^r\frac{f(\sigma_i)}{\sigma_i}{\bf v}_i{\bf v}_i^*\right), 
	\end{equation}
	or, equivalently, 
	\begin{equation}\label{eq:gen_fun2}
	\gf(A)=f(\sqrt{AA^*})(\sqrt{AA^*})^\dagger A=A(\sqrt{A^*A})^\dagger f(\sqrt{A^*A}).
	\end{equation}
	\end{subequations}
\end{theorem}

\begin{proof}
	The two identities are an easy consequence of the fact that 
        ${\bf u}_i = \frac{1}{\sigma_i} A {\bf v}_i$ and 
	${\bf v}_i = \frac{1}{\sigma_i}A^*{\bf u}_i$ for $i=1,2,\ldots,r$.
\end{proof}

\begin{proposition}
	Let $A\in\Cmn$ be a rank-$r$ matrix and let $f:\R\rightarrow\R$ and $g:\R\rightarrow\R$ 
        be two scalar functions such 
	that $\gf(A)$ and $g(AA^*)$ are defined. 
	Then 
	$$g(AA^*)\gf(A) = \gf(A)g(A^*A).$$
\end{proposition}

\begin{proof}
From $A = U_r\Sigma_r V_r^*$ it follows $AA^* = U_r \Sigma_r^2 U_r^*$ and $A^*A = V_r \Sigma_r^2 V_r^*$; thus 
\begin{align*}
g(AA^*)\gf(A) &= U_r g(\Sigma_r^2) U_r^* U_r f(\Sigma_r) V_r^* = U_r g(\Sigma_r^2) f(\Sigma_r) V_r^* \\
	      &= U_r f(\Sigma_r) g(\Sigma_r^2) V_r^* = U_r f(\Sigma_r) V_r^* V_r g(\Sigma_r^2) V_r^* \\
	      &= \gf(A) g(A^*A).
\end{align*}
\end{proof}


\section{Manifestations of generalized matrix functions}\label{sec:app}

As mentioned in the Introduction,
generalized matrix functions (in the sense of Hawkins and Ben-Israel)
have appeared in the literature 
without being recognized as such. 
Here we discuss a few examples that we are aware of. No doubt there have been
other such instances.

In \cite{DBLP05}, the authors address the problem of computing functions of
real skew-symmetric matrices, in particular the evaluation of the 
product $e^A b$ for a given skew-symmetric matrix $A$ and vector $b$ using the Lanczos
algorithm. The authors observe that any $A\in \R^{2n\times 2n}$ with $A^T = -A$
is orthogonally similar to a  matrix of the form
$$
\left(
\begin{array}{cc}
0 & -B\\
B^T & 0
\end{array}\right),
$$
where $B$ is lower bidiagonal of order $n$. As a consequence, if $B=U\Sigma V^T$ is an SVD of $B$,
the matrix exponential $e^A$ is orthogonally similar to the matrix
\begin{equation}\label{eq:Lopez}
\left(
\begin{array}{cc}
U\cos(\Sigma)U^T & -U\sin(\Sigma)V^T\\
V\sin(\Sigma) U^T & V\cos(\sigma)V^T
\end{array}\right),  
\end{equation}
where the matrix in the upper right block is precisely $-\sin^{\diamond}(B)$.
The authors of \cite{DBLP05} develop computational techniques
for the matrix exponential based on (\ref{eq:Lopez}). 
We also mention that in the same paper the authors derive
a similar expression, also found in \cite{BEK13}, for the exponential of 
the symmetric matrix $\mathscr{A}$ given in (\ref{eq:calA}).
These expressions are extended to more general matrix functions in
\cite{LopPug05}, where they are used to investigate the off-diagonal
decay of analytic functions of large, sparse, skew-symmetric matrices.
Furthermore, in \cite{DelLopPol06} it is shown how these ideas can be
used to develop efficient geometrical integrators for the numerical
solution of certain Hamiltonian differential systems.

In \cite{CEHT10}, the authors consider the problem of detecting (approximate)
directed bipartite communities in directed graphs. Consideration of alternating
walks in the underlying graph leads them to introducing a ``non-standard
matrix function" of the form
$$f(A) = I - A + \frac{AA^T}{2!} - \frac{AA^TA}{3!} + \frac{AA^TAA^T}{4!} - \cdots ,$$
where $A$ is the adjacency matrix of the graph. Using $A=U\Sigma V^T$ this expression
is readily recognized to be equivalent to 
$$f(A) = U\cosh (\Sigma) U^T - U \sinh (\Sigma) V^T,$$
which is a ``mixture" of the standard matrix function $\cosh (\sqrt{AA^T})$
and the generalized matrix function $\sinh^\diamond (A)$. 

As mentioned, generalized hyperbolic matrix functions were also considered
in \cite{BEK13} in the context of directed networks, also based on the notion
of alternating walks in directed graphs.     
In \cite{AB15}, the action of generalized matrix 
functions on a vector of all ones was used to define
certain centrality measures for nodes in directed graphs; here the
connection with the work of Hawkins and Ben-Israel was explicitly made.

Finally, we mention that generalized matrix functions arise when {\it filter factors}
are used to regularize discrete ill-posed problems; see, e.g.,
\cite{HNP}.

\section{Computational aspects}

The computation of the generalized matrix functions defined as in Definition \ref{def:gf}
requires the knowledge of the singular value decomposition of $A$. 
When $m$ and $n$ are large, computing the SVD may be unfeasible. 
Moreover, 
in most applications it is not required to compute the whole matrix $\gf(A)$;
rather, the goal is often to estimate
quantities of the form
\begin{equation}\label{gf_bilinear}
Z^T \gf(A) W, \qquad Z \in \R^{m \times k}, \;  W \in \R^{n \times k},
\end{equation}
or to compute the action of the generalized matrix function on a set 
of $k$ vectors, i.e., to evaluate $\gf(A)W$, usually with $k \ll \min\{m,n\}$.
For example, computing selected columns of $\gf(A)$ reduces to the evaluation of
$\gf(A)W$ where $W$ consists of the corresponding columns of the identity matrix $I_n$,
and computing selected entries of $\gf(A)$ requires evaluating
$Z^T \gf(A) W$ where $Z$ contains selected columns of the identity matrix $I_m$.

The problem of estimating or giving bounds on such quantities can be 
tackled, following~\cite{BookGM}, by using Gauss-type quadrature 
rules. 
As usual in the literature, we will first analyze the case $k=1$; 
the case of $k>1$ will be dealt with in section~\ref{sec:approach_bl}.

\subsection{Approximating ${\bf z^T}\bm{\gf(A)}{\bf w}$} 
It is known that in certain cases Gauss-type quadrature rules can be used to 
obtain lower and upper bounds on bilinear forms like ${\bf z}^T f(A)\bf w$, where $f(A)$
is a (standard) matrix function and $A=A^T$. 
This is the case when $f$ enjoys certain monotonicity
properties.  
Recall that a real-valued function $f$ is {\it completely monotonic} ({\it c.m.}) 
 on an interval $I = [a,b]$ if it is continuous on $I$ and infinitely differentiable 
on $(a,b)$ with
$$(-1)^nf^{(n)}(t)\geq 0,\qquad t\in (a,b),\; \forall n= 0,1,\ldots,$$
where $f^{(n)}$ denotes the $n$th derivative of $f$ and $f^{(0)}=f$.
If $f$ is completely monotonic on an interval containing the spectrum
of $A=A^T$, then one can obtain lower and upper bounds on quadratic forms of the
type ${\bf u}^T f(A) {\bf u}$ and from these lower and upper bounds on bilinear forms
like ${\bf z}^T f(A){\bf w}$ with ${\bf z}\ne {\bf w}$.
For a general $f$, on the other hand, Gaussian quadrature can only provide estimates 
of these quantities.

Similarly, in order to obtain bounds (rather than mere estimates) for bilinear expressions 
involving 
generalized matrix functions, we need the 
scalar functions involved in the computations to 
be completely monotonic. 

\begin{remark}
{\rm
	We will be applying our functions to diagonal matrices that contain the singular values of the matrix of interest. 
	Thus, in our framework, the interval on which we want to study the complete monotonicity of the 
	functions is $I=(0,\infty)$.
}
\end{remark}  

We briefly recall here a few properties of c.~m.~functions; see, e.g., \cite{MS01,W46} 
and references therein for systematic treatments of complete monotonicity. 

\begin{lemma}\label{lemma3}
	If $f_1$ and $f_2$ are completely monotonic functions on $I$, then 
	\begin{itemize}
		\item[(i)] $\alpha f_1(t) + \beta f_2(t)$ with $\alpha,\beta\geq 0$ is completely monotonic on $I$;
		\item[(ii)] $f_1(t)f_2(t)$ is completely monotonic on $I$.
	\end{itemize}
\end{lemma}

\begin{lemma}{\rm\cite[Theorem~2]{MS01}}\label{lemma1}
	Let $f_1$ be completely monotonic and let $f_2$ be a nonnegative function 
	such that $f_2'$ is completely monotonic. 
	Then $f_1\circ f_2$ is completely monotonic.
\end{lemma}

Using these lemmas, we can prove the following useful result.

\begin{theorem}
	If $f$ is completely monotonic on $(0,\infty)$, 
	then $g(t):=\frac{f(\sqrt{t})}{\sqrt{t}}$ is completely monotonic on $(0,\infty)$.
\end{theorem}
\begin{proof}
	Let $h(t) = t^{-1}$; then by Lemma~\ref{lemma3} (ii) we know that 
	$g(t)$ is completely monotonic on $I=(0,\infty)$ if both $f(\sqrt{t})$ and $h(\sqrt{t})$ are 
	completely monotonic on $I$. 
	The function $\sqrt{t}$ is positive on the interval $(0,\infty)$;  moreover, it 
	is such that its first derivative $\frac{1}{2}t^{-1/2}$ is completely monotonic on $I$. 
	Therefore, from Lemma~\ref{lemma1} it follows that if $f$ is c.m.~, then $f(\sqrt{t})$ is. 
	Similarly, since $h(t) = t^{-1}$ is completely monotonic, $h(\sqrt{t})$ is completely monotonic.  
	This concludes the proof.
\end{proof}

In the following, we propose three different approaches to approximate the bilinear forms of interest. 
The first approach exploits the results of Theorem~\ref{thm:genfun} to describe ${\bf z}^T\gf(A){\bf w}$ as a bilinear form that involves 
standard matrix functions of a tridiagonal matrix. 
The second approach works directly with the generalized matrix function and the Moore--Penrose pseudo-inverse of a bidiagonal matrix. 
The third approach first approximates the action of a generalized matrix function on a vector and then derives the approximation for the bilinear form of interest.

\subsection{First approach}\label{ssec:approach1}

When the function $f:\R\longrightarrow\R$ that defines $\gf$ is c.m., then 
Gauss-type quadrature rules can be used to derive upper and lower bounds for the quantities of interest.  
It is straightforward to see by using \eqref{eq:gen_fun1} that a bilinear form involving 
a generalized matrix function can be written as
$$
{\bf z}^T \gf(A) {\bf w} = 
{\bf z}^T \left(\sum_{i = 1}^r \frac{f(\sigma_i)}{\sigma_i}{\bf u}_i{\bf u}_i^T\right) \wt{{\bf w}} = 
\wt{{\bf z}}^T \left(\sum_{i = 1}^r \frac{f(\sigma_i)}{\sigma_i}{\bf v}_i{\bf v}_i^T\right) {\bf w},
$$
where $\wt{{\bf w}} = A {\bf w}$, and $\wt{{\bf z}} = A^T {\bf z}$. 
Using the equalities in \eqref{eq:gen_fun2} one can see that these quantities can also be expressed as 
bilinear forms involving functions of the matrices $AA^T$ and $A^TA$, respectively. 
More in detail, one obtains 
\begin{subequations}
\begin{equation}\label{bilinear}
{\bf z}^T \gf(A) {\bf w} = 
\wt{{\bf z}}^T \left(\sum_{i = 1}^r\frac{f(\sigma_i)}{\sigma_i}{\bf v}_i{\bf v}_i^T\right) {\bf w} =
\wt{{\bf z}}^T g(A^TA) {\bf w},
\end{equation}
 
\begin{equation}\label{bilinear2}
{\bf z}^T \gf(A) {\bf w} = 
{\bf z}^T \left(\sum_{i=1}^r\frac{f(\sigma_i)}{\sigma_i}{\bf u}_i{\bf u}_i^T\right) \wt{{\bf w}} = 
{\bf z}^T g(AA^T) \wt{{\bf w}},
\end{equation}
\end{subequations}
where in both cases $g(t)= (\sqrt{t})^{-1} f(\sqrt{t})$. 

In the following we
focus on the case described by \eqref{bilinear}. 
The discussion for the case described by \eqref{bilinear2} follows the same lines.

\begin{remark}
{\rm Note that if ${\bf z},{\bf w}$ are vectors such that $\wt{{\bf z}} \neq {\bf w}$, then 
we can use the {\it polarization identity}~\cite{BookGM}:
\begin{equation*}
\wt{{\bf z}}^T g(A^TA) {\bf w} = \frac{1}{4}\left[(\wt{{\bf z}}+{\bf w})^T  g(A^TA)(\wt{{\bf z}}+{\bf w}) - 
(\wt{{\bf z}} - {\bf w})^T  g(A^TA) (\wt{{\bf z}} - {\bf w})\right] 
\end{equation*}
to reduce the evaluation of the bilinear form of interest to the evaluation of two symmetric bilinear forms. 
For this reason, the theoretical description of the procedure to follow will be carried out only for the case $\wt{{\bf z}} = {\bf w}$.}
\end{remark}

Let $\wt{{\bf z}} = {\bf w}$ be a unit vector (i.e., $\|{\bf w}\|_2 = 1$). 
We can rewrite the quantity \eqref{bilinear} as a Riemann--Stieltjes 
integral by substituting the spectral factorization of $A^TA$:
\begin{equation}\label{stieltjes}
{\bf w}^T g(A^TA) {\bf w} = {\bf w}^TV_r g(\Sigma_r^2)V_r^T {\bf w} = 
\sum_{i=1}^{r} \frac{f(\sigma_i)}{\sigma_i} ({\bf v}_i^T{\bf w})^2 = 
\int_{\sigma_r^2}^{\sigma_1^2}  g(t) \; d\alpha(t),
\end{equation}
where $\alpha(t)$ is a piecewise constant step function with jumps at the positive eigenvalues 
$\{\sigma_i^2\}_{i=1}^r$ of $A^TA$ defined as follows:
$$\alpha(t) = 
\left\{
\begin{array}{ll}
0,                                     & \text{ if } t<\sigma_r^2 \\
\sum_{i=j+1}^r ({\bf v}_i^T{\bf w})^2, & \text{ if } \sigma_{j+1}^2\leq t < \sigma_j^2 \\
\sum_{i=1}^r ({\bf v}_i^T{\bf w})^2,   & \text{ if } t\geq\sigma_r^2.
\end{array}
\right. $$

We use partial Golub--Kahan bidiagonalization \cite{GK65,GVL} of the matrix $A$ 
to find upper and lower bounds for the bilinear form described in~\eqref{stieltjes}. 
After $\ell$ steps, the Golub--Kahan bidiagonalization of the matrix $A$ with initial vector ${\bf w}$ yields the decompositions
\begin{equation}\label{lbd}
AQ_{\ell}=P_{\ell}B_{\ell},\qquad A^TP_{\ell}=Q_{\ell}B_{\ell}^T + \gamma_{\ell} {\bf q}_{\ell} {\bf e}_{\ell}^T,
\end{equation}
where the matrices $Q_{\ell}=[{\bf q}_0,{\bf q}_1,\ldots,{\bf q}_{\ell-1}]\in\R^{n\times\ell}$ 
and $P_{\ell} = [{\bf p}_0,{\bf p}_1,\ldots,{\bf p}_{\ell-1}]\in\R^{m\times\ell}$
have orthonormal columns, the matrix
$$
B_{\ell}=\left(
\begin{array}{cccc}
\omega_1 & \gamma_1 &                  &                 \\
	 &  \ddots  & \ddots 	       &                 \\
         &  	    & \omega_{\ell-1}  & \gamma_{\ell-1} \\
         &          & 	      	       & \omega_{\ell}   \\          
\end{array}
\right)
\in\R^{\ell\times\ell}$$
is upper bidiagonal, and the first column of $Q_{\ell}$ is ${\bf w}$.

\begin{remark}
{\rm
All the $\{\gamma_j\}_{j=1}^{\ell-1}$ and $\{\omega_j\}_{j=1}^\ell$ can be assumed to be nonzero~\cite{GK65}.
With this assumption, the CSVD of the bidiagonal matrix $B_\ell$ coincides with its SVD: 
$$B_\ell = {\pcal U}_\ell\Theta_\ell{\pcal V}_\ell^T,$$ 
where ${\pcal U}_\ell = [\bm{ \upsilon}_1,\bm{ \upsilon}_2,\ldots,\bm{ \upsilon}_\ell]\in\R^{\ell\times\ell}$ and 
${\pcal V}_\ell = [\bm{\nu}_1,\bm{ \nu}_2,\ldots,\bm{\nu}_\ell]\in\R^{\ell\times\ell}$ are orthogonal, and 
${\Theta}_\ell = \diag(\theta_1,\theta_2,\ldots,\theta_\ell)\in\R^{\ell\times\ell}$.
}
\end{remark}

Combining the equations in~\eqref{lbd} leads to
\begin{equation*}
A^TA Q_{\ell}=Q_{\ell} B_{\ell}^T B_{\ell} + \gamma_{\ell}\omega_{\ell} {\bf q}_{\ell}{\bf e}_{\ell}^T,
\end{equation*}
where ${\bf q}_{\ell}$ denotes the Lanczos vector computed at iteration $\ell + 1$ . 
The matrix
\begin{equation*}
T_{\ell}=B_{\ell}^T B_{\ell}
\end{equation*}
is thus symmetric and tridiagonal and coincides (in exact arithmetic)
with the matrix obtained when the Lanczos algorithm  is applied to $A^TA$.

The quadratic form in \eqref{stieltjes} can then be approximated by using 
an $\ell$-point Gauss quadrature rule~\cite{BookGM}:
\begin{equation}\label{eq:Gell}
{\pcal G}_\ell := {\bf e}_1^T g(T_\ell){\bf e}_1 = {\bf e}_1^T(\sqrt{T_\ell})^\dagger f(\sqrt{T_\ell}){\bf e}_1.
\end{equation} 

If the function $f(t)$ is c.m., then
the Gauss rule provides a lower bound for~\eqref{stieltjes}, which can be shown to be strictly increasing with $\ell$.
If the recursion formulas for the Golub--Kahan bidiagonalization break
down, that is, if $\gamma_{\ell} = 0$ at step $\ell$, then the Gauss quadrature rule gives the exact value (see~\cite{GVL}).

The following result can be easily derived from equation~\eqref{eq:Gell}.

\begin{proposition}\label{prop:nodes}
Let $A \in \R^{m \times n}$ and let $B_{\ell} \in \R^{\ell \times \ell}$ be the bidiagonal matrix computed after $\ell$ steps of the Golub--Kahan bidiagonalization algorithm. 
Let $(\theta_i,\bm{\upsilon}_i,\bm{\nu}_i)$ for $i=1,2,\ldots,\ell$ be the singular triplets of 
$B_\ell = {\pcal U}_\ell \Theta_\ell {\pcal V}_\ell^T$. 
Then the nodes of the $\ell$-point Gauss quadrature rule ${\pcal G}_\ell$ are the singular values $\{\theta_i\}_{i=1}^\ell$. 
Furthermore, if ${\bf z}=\wt{{\bf w}}$, the weights of ${\pcal G}_\ell$ are $({\bf e}_1^T{\bm{\upsilon}_i})^2\theta_i^{-1}$ 
for $i=1,2,\ldots,\ell$. 

Similarly, if $\wt{{\bf z}} = {\bf w}$, then the weights of the rule are given by $({\bf e}_1^T{\bm{\nu}_i})^2\theta_i^{-1}$.
\end{proposition}

To provide an upper bound for~\eqref{stieltjes} when $f$ is c.~m., 
one can use a $(\ell +1)$-point Gauss--Radau quadrature rule 
with a fixed node $\tau = \sigma_1^2$; this 
can be expressed in terms of the entries of the symmetric tridiagonal matrix 
$$
\widehat{T}_{\ell+1} = \left(
\begin{array}{cc}
T_{\ell}                    &   \rho_\ell{\bf e}_\ell    \\
\rho_\ell{\bf e}_\ell^T   &   \widehat{\omega}_{\ell+1} \\
\end{array}
\right)
\in \R^{(\ell+1)\times(\ell+1)}
$$
as $\widehat{{\pcal G}}_{\ell+1}:={\bf e}_1^Tg(\widehat{T}_{\ell+1}){\bf e}_1$, where $g(t) = (\sqrt{t})^{-1}f(\sqrt{t})$. 
The entries of this matrix, except for the last 
diagonal entry, are those of $B_{\ell+1}^TB_{\ell+1}$. 
To compute the last diagonal entry so that $\widehat{T}_{\ell+1}$ 
has $\tau = \sigma_1^2$ among its eigenvalues, we proceeds as follows~\cite{BookGM}. 
First, we compute $\rho_\ell$; then we set $\widehat{\omega}_{\ell+1} = 
\tau + {\bf e}_\ell^T{\bf x}$, where ${\bf x}$ is the solution 
of the tridiagonal linear system $(T_\ell - \tau I){\bf x} = \rho_\ell^2{\bf e}_\ell$.  
The arithmetic mean between the $\ell$-point Gauss rule ${\pcal G}_\ell$ and the 
$(\ell +1)$-point Gauss--Radau rule $\widehat{{\pcal G}}_{\ell+1}$ is then used 
as an approximation of the quadratic form ${\bf w}^Tg(A^TA){\bf w}$.


\subsection{Second approach}\label{ssec:approach2}
In this section we provide a second approach to the approximation of bilinear forms expressed in terms of generalized matrix functions.

The following result shows how to compute the $\ell$-point Gauss quadrature 
rule in terms of the generalized matrix function of the bidiagonal matrix $B_\ell$. 
Two expressions are derived, depending on the starting (unit) vector given as input to the Golub--Kahan algorithm.
Recall that, unless ${\bf z}=A{\bf w}$ or ${\bf w} = A^T{\bf z}$, 
one has to use the polarization identity to estimate the bilinear forms of interest.

\begin{proposition}
Let be $A \in \R^{m \times n}$ and let $B_{\ell} \in \R^{\ell \times \ell}$ be the bidiagonal matrix computed at step $\ell$ of the Golub--Kahan bidiagonalization algorithm. 
Then, the $\ell$-point Gauss quadrature rule ${\pcal G}_\ell$ is given by  
$$
{\pcal G}_\ell = {\bf e}_1^T B_\ell^{\dagger} \gf(B_\ell){\bf e}_1, \quad \text{ if }\ \wt{{\bf z}} = {\bf w},
$$
or 
$$
{\pcal G}_\ell = {\bf e}_1^T \gf(B_\ell)B_\ell^{\dagger} {\bf e}_1, \quad \text{ if }\  {\bf z} = \wt{{\bf w}}.
$$ 
\end{proposition}
\begin{proof}
Let $B_\ell = {\pcal U}_\ell\Theta_\ell {\pcal V}_\ell^T$ be a singular value decomposition of the matrix $B_\ell$ obtained after $\ell$ steps 
of the Golub--Kahan bidiagonalization algorithm with starting vector $\wt{{\bf z}} = {\bf w}$. 
Then, from Proposition~\ref{prop:nodes}, it follows that
$$
\begin{aligned}
{\pcal G}_\ell &=  \sum_{i = 1}^{\ell} f(\theta_i ) \frac{({\bf e}_1^T \bm{\nu}_i)^2}{\theta_i} 
		  = {\bf e}_1^T\left( \sum_{i = 1}^{\ell} \frac{f(\theta_i)}{\theta_i}\bm{\nu}_i\bm{\nu}_i^T\right){\bf e}_1  \\
   		 &= {\bf e}_1^T {\pcal V}_\ell \Theta_\ell^{\dagger} f(\Theta_\ell) {\pcal V}_\ell^T {\bf e}_1 
		  =  {\bf e}_1^T B_\ell^{\dagger} {\pcal U}_\ell f(\Theta_\ell) {\pcal V}_\ell^T {\bf e}_1 \\  
		 &={\bf e}_1^T B_\ell^{\dagger} \gf(B_\ell){\bf e}_1.
\end{aligned}
$$
The proof of the case when ${\bf z} = \wt{{\bf w}}$ goes along the same lines and it is thus omitted.
\end{proof}

The $(\ell+1)$-point Gauss-Radau quadrature rule $\widehat{\pcal G}_{\ell+1}$ with a fixed node $\sigma_1$ can be expressed in terms of the entries of the bidiagonal matrix 
$$
\widehat{B}_{\ell+1} =\left( 
\begin{array}{cc} 
B_\ell       & \gamma_\ell \mathbf{e}_\ell \\ 
\mathbf{0}^T & \widehat{\omega}_{\ell + 1}
\end{array}
\right) \in \mathbb{R}^{(\ell+1) \times (\ell+1)}
$$
as $\widehat{\pcal G}_{\ell+1} = {\bf e}_1^T \widehat{B}_{\ell+1}^{\dagger}f^{\diamond}(\widehat{B}_{\ell+1}) {\bf e}_1$ 
if $\widetilde{{\bf z}} = {\bf w}$ or as $\widehat{\pcal G}_{\ell+1} = {\bf e}_1^T f^{\diamond}(\widehat{B}_{\ell+1})\widehat{B}_{\ell+1}^{\dagger} {\bf e}_1$ when ${\bf z} = \widetilde{{\bf w}}$.

The entries of $\widehat{B}_{\ell+1}$, except for the last diagonal entry, 
are those of $B_{\ell+1}$. To compute the last diagonal entry, one has to 
ensure that $\sigma_1^2$ is an eigenvalue of 
$\widehat{T}_{\ell+1} = \widehat{B}_{\ell+1}^T\widehat{B}_{\ell+1}$. 
It can be easily shown that 
$$
\widehat{\omega}_{\ell + 1} = \sqrt{\sigma_1^2 + \mathbf{e}_\ell^T\mathbf{x} - \gamma_\ell^2},
$$
where $\mathbf{x}$ is the solution of the tridiagonal linear system 
$(B_\ell^T B_\ell - \sigma_1^2 I) \mathbf{x} = (\omega_\ell \gamma_\ell)^2 \mathbf{e}_\ell$.


\subsection{Third approach}\label{ssec:approach3}
Assume that we have used $\ell = r = \rank(A) $ 
steps of the Golub--Kahan bidiagonalization algorithm 
with starting vector ${\bf w}$ (normalized so as to have unit norm) 
to derive the matrices $P_r$, $B_r$, and $Q_r$ such that $A = P_r B_r Q_r^T$. 
The CSVD of the bidiagonal matrix is $B_r ={\pcal U}_r \Sigma_r {\pcal V}_r^T$, where $\Sigma_r$ is the same diagonal matrix appearing in the CSVD of $A$. 
Since $P_r$ and $Q_r$ have full column rank, we know that $\rank(P_r B_r Q_r^T) = \rank(B_r) = r$, and thus we 
can write
\begin{equation*}
\begin{split}
{\bf z}^T\gf(A){\bf w} &= {\bf z}^T\gf(P_r B_r Q_r^T){\bf w} 
			= {\bf z}^T\gf(P_r {\pcal U}_r \Sigma_r {\pcal V}_r^T Q_r^T){\bf w} \\
                       &= {\bf z}^T(P_r {\pcal U}_r)f(\Sigma_r)(Q_r {\pcal V}_r)^T {\bf w} 
			= \widehat{{\bf z}}^T \gf(B_r) {\bf e}_1,
\end{split}
\end{equation*}
where $\widehat{{\bf z}} = P_r^T {\bf z}$ and $Q_r^T{\bf w} = {\bf e}_1$. 

Assume now that $\ell < r$. 
We can then truncate the bidiagonalization process and approximate $\gf(A){\bf w}$ as 
$$\gf(A){\bf w}\approx P_\ell \gf(B_\ell){\bf e}_1$$
and then obtain the approximation to the bilinear form of interest as  
$${\bf z}^T\gf(A){\bf w} \approx {\bf z}^TP_\ell\gf(B_\ell){\bf e}_1.$$
The quality of the approximation will depend in general on the distribution of the
singular values of $A$ and on the particular choice of $f$.
Generally speaking, if $f(\sigma_i)$ is much larger on the first few singular
values of $A$ than for the remaining ones, then a small number
of steps result in approximations with small relative errors.

\section{The block case}\label{sec:approach_bl} 

In this section we describe two ways to compute approximations of quantities of 
the form~\eqref{gf_bilinear}, when $k>1$. It is known that for this kind of problem,
block algorithms are  generally more efficient than the separate computation of each individual
entry (or column) of $Z^T\gf(A)W$.

When dealing with blocks $Z$ and $W$ with a number of columns $k > 1$, 
the complete monotonicity of the function $f$ does not ensure that block 
Gauss-type quadrature rules provide bounds on the quantities of the form $Z^T\gf(A)W$. 
In this case, indeed, no information about the sign of the quadrature error can 
be obtained from the remainder formula for the Gauss quadrature rules~\cite{BookGM}. 
Therefore, we focus on the computation of approximations for the quantities of 
interest, rather than on bounds. 
We propose two different approaches to compute the quantities~\eqref{gf_bilinear}. 
The first one exploits the connection between generalized matrix functions and standard 
matrix functions described in 
Theorem~\ref{thm:genfun}, while the second one first approximates the action 
of a generalized matrix function on $k$ vectors and then derives the approximation of the quantities of interest.

\subsection{First approach}\label{ssec:block1}
As a first approach, we propose the use of a pair of block Gauss and anti-Gauss quadrature 
rules~\cite{La,CRS,FMRR} based on the nonsymmetric block Lanczos algorithm~\cite{BookGM}.  
As already pointed out, if we let $g(t) = (\sqrt{t})^{-1}f(\sqrt{t})$, it holds that
$$
Z^T f^{\diamond}(A) W = \widetilde{Z}^T g(A^TA)W = Z^T g(AA^T)\widetilde{W},
$$ 
where $\widetilde{Z} = A^TZ$ and $\widetilde{W} = AW$. 
In this case, there is no equivalent to the polarization identity and thus we 
work directly with the blocks $\widetilde{Z}$ and $W$, $\widetilde{Z} \neq W$ 
(the case when $Z$ and $\widetilde{W}$ are the initial blocks is similar).
 
Let $\widetilde{Z}_0\Delta_0^T \in\R^{n \times k}$ and $W_0\Gamma_0^T\in\R^{m\times k}$ 
have all zero entries. 
Assume moreover that $\widetilde{Z}_1 = \widetilde{Z}$ and $W_1 = W$ satisfy 
$\widetilde{Z}_1^TW_1 = I_k$.
Then the nonsymmetric block Lanczos algorithm
 applied to the matrix $X = A^TA$ is described by the following recursions:
\begin{equation}\label{xrecur2}
\begin{aligned}
&\Omega_j=W_{j}^T\left(X\widetilde{Z}_j-\widetilde{Z}_{j-1}\Delta_{j-1}^T\right), \\
&R_j=X\widetilde{Z}_j-\widetilde{Z}_j\Omega_j-\widetilde{Z}_{j-1}\Delta_{j-1}^T, 
\qquad &Q_RR_R&=R_j, \\
&S_j=X^TW_j-W_j\Omega_j^T-W_{j-1}\Gamma_{j-1}^T, \qquad &Q_SR_S&=S_j,\\
&Q_S^TQ_R=\widetilde{U}\widetilde{\Sigma} \widetilde{V}^T, \\
&\widetilde{Z}_{j+1}=Q_R\widetilde{V}\widetilde{\Sigma}^{-1/2}, 
\quad W_{j+1}=Q_S\widetilde{U}\widetilde{\Sigma}^{-1/2}, \\
&\Gamma_j=\widetilde{\Sigma}^{1/2}\widetilde{V}^TR_R, \quad 
\Delta_j=\widetilde{\Sigma}^{1/2}\widetilde{W}^TR_S,
\end{aligned}
\end{equation}
$j = 1, \dots, \ell$. In~\eqref{xrecur2}, $Q_RR_R=R_j$ and $Q_SR_S=S_j$ are the QR factorizations of $R_j$ and $S_j$, respectively, and $\widetilde{U}\widetilde{\Sigma} \widetilde{V}^T$ is a singular value
decomposition of the matrix $Q_S^TQ_R$. The recursion formulas~\eqref{xrecur2} ensures that $\widetilde{Z}_j^TW_j = \delta_{ij}I_k$.

More succinctly, after $\ell$ steps, the nonsymmetric block Lanczos algorithm applied to the 
matrix $X = A^TA$ with initial blocks $\widetilde{Z}_1$ and $W_1$ yields the decompositions
\begin{equation*}
\begin{aligned}
X\left[\widetilde{Z}_1,\dots,\widetilde{Z}_\ell\right]
				&=\left[\widetilde{Z}_1,\dots,\widetilde{Z}_\ell\right]J_\ell+\widetilde{Z}_{\ell+1}\Gamma_\ell 					\bm{E}_\ell^T,\\
X\left[W_1,\dots,W_\ell\right]	&=\left[W_1,\dots,W_\ell\right]J_\ell^T+W_{\ell+1}\Delta_\ell \bm{E}_\ell^T,
\end{aligned}
\end{equation*}
where $J_\ell$ is the matrix 
\begin{equation}\label{JL}
J_\ell=\left(
\begin{array}{ccccc}
\Omega_1&\Delta_1^T&&&\\
\Gamma_1&\Omega_2&\Delta_2^T&&\\

&\ddots&\ddots&\ddots&\\
&&\Gamma_{\ell-2}&\Omega_{\ell-1}&\Delta_{\ell-1}^T\\
&&&\Gamma_{\ell-1}&\Omega_\ell

\end{array}
\right)\in\R^{k\ell\times k\ell},
\end{equation}
and $\bm{E}_i$, for $i=1,2,\ldots,\ell$ are $k\times (k\ell)$ block matrices 
which contain $k\times k$ zero blocks everywhere, except for the $i$th block, 
which coincides with the identity matrix $I_k$.
We remark that if $\widetilde{Z} = W$, the use of the symmetric block 
Lanczos algorithm is preferable.  In this case, the matrix $J_\ell$~\eqref{JL} is symmetric and the decompositions~\eqref{xrecur2} can be written as 
$$
X\left[W_1,\dots,W_\ell\right]=\left[W_1,\dots,W_\ell\right]J_\ell+W_{\ell+1}\Gamma_\ell \bm{E}_\ell^T.
$$
The $\ell$-block nonsymmetric Gauss
quadrature rule ${\pcal G}_\ell$ can then be expressed as
\begin{equation*}
{\pcal G}_\ell = \bm{E}_1^Tg(J_\ell) \bm{E}_1.
\end{equation*}

The $(\ell+1)$-block anti-Gauss quadrature rule ${\pcal H}_{\ell+1}$ is defined as the $(\ell+1)$-block quadrature rule such that
\begin{equation*}
\left({\pcal I}-{\pcal H}_{\ell+1}\right)p=-\left({\pcal I}-{\pcal G}_\ell\right)p,
\qquad p\in\mathbb{P}^{2\ell+1},
\end{equation*}
where 
${\pcal I}p:= \sum_{i=1}^r p(\sigma_i^2)\bm{\alpha}_i\bm{\beta}_i^T + \sum_{i=r+1}^n p(0)
\bm{\alpha}_i\bm{\beta}_i^T
= (\wt{Z}^TV)p(\Sigma^T\Sigma)(W^T V)^T$, with
$\wt{Z}^TV = [\bm{\alpha}_1,\bm{\alpha}_2,\ldots,\bm{\alpha}_n]$,
$W^T V = [\bm{\beta}_1,\bm{\beta}_2,\ldots,\bm{\beta}_n]\in\mathbb{R}^{k\times n}$, 
and $\mathbb{P}^{2\ell+1}$ is the set of polynomials of degree at most $2\ell+1$
(see~\cite{BookGM}).
As shown in \cite{FMRR}, the
$(\ell+1)$-block nonsymmetric anti-Gauss rule can be computed in terms of 
the matrix $\widetilde{J}_{\ell+1}$ as
\begin{equation*}
{\pcal H}_{\ell+1}= \bm{E}_1^Tg(\widetilde{J}_{\ell+1}) \bm{E}_1,
\end{equation*}
where 
$$
\widetilde{J}_{\ell+1}=\left(\begin{array}{ccccc}
\Omega_1&\Delta_1^T&&&\\
\Gamma_1&\Omega_2&\Delta_2^T&&\\
&\ddots&\ddots&\ddots&\\
&&\Gamma_{\ell-1}&\Omega_{\ell}&\sqrt{2}\Delta_{\ell}^T\\
&&&\sqrt{2}\Gamma_{\ell}&\Omega_{\ell+1}
\end{array}\right)\in\R^{k(\ell+1)\times k(\ell+1)}.
$$

A pair of block Gauss and anti-Gauss quadrature rules is not guaranteed 
to provide upper and lower bounds, not even in the case $k = 1$. 
However, suppose that the function $g$ can be written as
$$
g(t) = \sum_{i=1}^\ell \eta_i p_i(t),
$$
where $p_i(t)$ are the orthonormal polynomials implicitly defined by the scalar Lanczos algorithm.
In~\cite{CRS}, the authors show that if the coefficients $\eta_i$ decay rapidly to zero, then 
$$
\left({\pcal I}-{\pcal H}_{\ell+1}\right)g \approx -\left({\pcal I}-{\pcal G}_\ell\right)g,
$$
that is, a pair of scalar Gauss and anti-Gauss rules provides {\it estimates} of
upper and lower bounds on the bilinear form of interest. 
This result has been extended to the block case in~\cite{FMRR}. 
In this framework, if we express $g(t)$ in terms of orthonormal polynomials, the coefficients in the expansion are $k \times k$ matrices. 
To obtain good entrywise approximations for the quantities of interest it is necessary that the norm of the coefficients decays rapidly as $\ell$ increases. 
This condition is satisfied if $g(t)$ is analytic in a 
simply connected domain $D$ enclosing the spectrum of $A^TA$, 
as long as the boundary $\partial D$ is not close to the spectrum \cite{FMRR}.

If the function $g(t)$ satisfies the above conditions, the arithmetic mean 
\begin{equation}\label{FL}
F_\ell = \frac{1}{2}\left({\pcal G}_\ell + {\pcal H}_{\ell+1}\right)
\end{equation} 
between Gauss and anti-Gauss quadrature rules can be used as an approximation of the 
matrix-valued expression $Z^T f^{\diamond}(A) W$.


\subsection{Second approach}\label{ssec:block2}
The second approach extends to the block case the approach described in subsection~\ref{ssec:approach3}. 
Assume that the initial block $W \in \mathbb{R}^{n \times k}$ satisfies $W^TW = I_k$ and that the matrices 
$\Gamma_0 \in \mathbb{R}^{k \times k}$ and $P_0 \in \mathbb{R}^{m \times k}$ are zero matrices. 
The following recursions determine the first $\ell$ steps of the block Golub--Kahan algorithm with starting block $Q_1 = W$:
\begin{equation}\label{recur}
\begin{aligned}
&R_j=AQ_j-P_{j-1}\Gamma_{j-1}^T, \\
&P_j\Omega_j=R_j,\\
&S_j =A^TP_j-Q_j\Omega_j^T, \\
&Q_{j+1}\Gamma_j=S_j, \\
\end{aligned}
\qquad j=1, \dots, \ell,
\end{equation}
where $P_j\Omega_j = R_j$ and $Q_{j+1}\Gamma_j = S_j$ are QR factorizations of $R_j$ and $S_j$, respectively.

After $\ell$ steps, the recursions~\eqref{recur} yield the decompositions
\begin{equation*}
\begin{aligned}
A\left[Q_1,\dots,Q_\ell\right]&=\left[P_1,\dots,P_\ell\right]B_\ell,
\\
A^T\left[P_1,\dots,P_\ell\right]&=\left[Q_1,\dots,Q_\ell\right]B_\ell^T+Q_{\ell+1}\Gamma_\ell  \bm{E}_\ell^T,
\end{aligned}
\end{equation*}
where now  
$$
B_\ell=\left(\begin{array}{ccccc}
\Omega_1&\Gamma_1^T&&&\\
        &\Omega_2&\Gamma_2^T&&\\
        &        &\ddots&\ddots&\\
        &        &      &\Omega_{\ell-1}&\Gamma_{\ell-1}^T\\
&&&&\Omega_\ell
\end{array}\right)\in\R^{k\ell\times k\ell}.
$$
Following the same reasoning as in subsection~\ref{ssec:approach3}, when $k\ell < r = \rank{(A)}$, 
we can approximate the quantities of interest as
$$ Z^T f^{\diamond}(A) W \approx Z^T [P_1, \dots, P_\ell] f^{\diamond}(B_\ell)  \bm{E}_1 = F_\ell.
$$

\section{Numerical results}

In this section we present some numerical results concerning the application 
of the previously introduced techniques 
to the computation of centrality and communicability indices in directed networks.
The first set of experiments concerns the computation of the 
{\em total hub communicability} of nodes, 
which, for a node $i$, is defined as the following bilinear form:
\begin{equation}\label{eq:tc}
C_h (i) := [\sinh^\diamond(A){\bf 1}]_i = {\bf e}_i^T \sinh^\diamond(A) {\bf 1}\,,
\end{equation}
where $A$ is the adjacency matrix of the digraph.
As shown in \cite{AB15}, this quantity can be used to rank how important node $i$
is when regarded as a ``hub", i.e., as a
broadcaster of information (analogous quantities rank the nodes
in order of their importance as ``authorities", i.e.,  receivers of information).
The second set of experiments concerns the computation of the resolvent-based 
communicability~\cite{BEK13} 
between node $i$, playing the 
role of broadcaster of information, and node $j$, acting as a receiver. 
The quantities of interest here have the form $[\gh(A)]_{ij}$, where $h(t)=\alpha t(1-(\alpha t)^2)^{-1}$ and $\alpha\in(0,\sigma_1^{-1})$. 
In all the tests we apply the approaches previously described and we use as stopping criterion 
\begin{equation}\label{eq:sc}
{\pcal R}_\ell = \left|\frac{x^{(\ell+1)}-x^{(\ell)}}{x^{(\ell)}}\right|\leq{\rm tol},
\end{equation}
where ${\rm tol}$ is a fixed tolerance and $x^{(\ell)}$ represents the approximation to the bilinear form of interest computed 
at step $\ell$ by the method under study. 

Our dataset contains the adjacency matrices associated with three real world unweighted and directed networks: \textsf{Roget}, \textsf{SLASHDOT}, 
and \textsf{ITwiki}~\cite{DataP,DataF,DataM}. 
The adjacency matrix associated with \textsf{Roget} is $994\times 994$ and has 7281 nonzeros.
The graph contains information concerning the cross-references in Roget's Thesaurus.
The adjacency matrix associated with \textsf{SLASHDOT} is an $82168\times 82168$ matrix with $948464$ nonzeros. 
For this network, there is a connection from node $i$ to node $j$  if user $i$ indicated user $j$ as a friend or a foe. 
The last network used in the tests, \textsf{ITwiki}, represents the Italian Wikipedia. 
Its adjacency matrix is $49728\times 49728$ and has $941425$ nonzeros, 
and there is a link from node $i$ to node $j$ in the graph if page $i$ refers to page $j$. 

\subsection*{Node centralities}
In this section we want to investigate how the three approaches defined for the 
case of $k=1$ perform when we want to approximate (\ref{eq:tc}), the total communicability 
of nodes in the network. 
For each network in the dataset, we computed the centralities of ten nodes chosen 
uniformly at random 
among all the nodes in the graph. 


The results for the tests are presented in Tables~\ref{tab:Roget}-\ref{tab:itwiki2}. 
The tolerance used in the stopping criterion \eqref{eq:sc} is set to ${\rm tol}=10^{-6}$.
The tables display the number of iterations required to satisfy the above criterion 
and the relative error of the 
computed solution with respect to the ``exact" value of the bilinear form. The
latter has been computed
using the full SVD for the smallest network, and using a partial SVD with a sufficiently 
large number of terms $(\gg \ell)$ for the two larger ones. The relative error is
denoted by
$${\pcal E}_\ell = \frac{|x^{(\ell)}- {\bf z}^T\gh(A){\bf w}|}{|{\bf z}^T\gh(A){\bf w}|}.$$ 

Concerning the first approach, since $g(t) = (\sqrt{t})^{-1}\sinh(\sqrt{t})$ is not 
completely monotonic, we have used the Gauss quadrature rule as an approximation for the quantities of interest, rather 
than as a lower bound.

\begin{table}[h!]
\footnotesize
\centering
\caption{{\rm Network: \textsf{Roget}, $h(t) = \sinh(t)$ (${\rm tol} = 10^{-6}$).}}
\label{tab:Roget}
\begin{tabular}{c|cccccc}
\hline
                     	      &\multicolumn{2}{c}{First approach} &\multicolumn{2}{c}{Second approach}  &\multicolumn{2}{c}{Third approach}\\
                              & ITER	& ${\pcal E}_\ell$	  & ITER	 & ${\pcal E}_\ell$	& ITER & ${\pcal E}_\ell$    \\

\hline

1                             & 8  &  1.10e-06  	      & 8  &  1.10e-06  		& 9   &  9.45e-09 \\
2                             & 34 &  9.93e-08  	      & 34 &  9.74e-08  		& 10  &  2.78e-09 \\
3                             & 5  &  3.20e-05  	      & 5  &  3.20e-05  		& 8   &  5.26e-07 \\
4                             & 6  &  4.38e-06  	      & 6  &  4.38e-06  		& 9   &  1.21e-08 \\
5                             & 20 &  6.18e-06  	      & 20 &  6.18e-06  		& 9   &  1.21e-08 \\
6                             & 7  &  2.62e-06  	      & 7  &  2.62e-06  		& 10  &  3.68e-10 \\
7                             & 8  &  7.08e-06  	      & 8  &  7.08e-06  		& 9   &  1.99e-08 \\
8                             & 15 &  9.07e-07  	      & 15 &  9.07e-07  		& 9   &  2.80e-08 \\
9                             & 9  &  8.15e-08  	      & 9  &  8.15e-08  		& 9   &  1.72e-09 \\
10                            & 7  &  3.78e-07  	      & 7  &  3.78e-07  		& 9   &  2.64e-08 \\

\hline

\end{tabular}
\end{table}
\begin{table} [h!]
\footnotesize
\centering
\caption{{\rm Network: \textsf{SLASHDOT}, $h(t) = \sinh(t)$ (${\rm tol} = 10^{-6}$).}}
\label{tab:slashdot2}
\begin{tabular}{c|cccccc}
\hline
                       & \multicolumn{2}{c}{First approach}  & \multicolumn{2}{c}{Second approach} & \multicolumn{2}{c}{Third approach} \\
                       & ITER & ${\pcal E}_\ell$	     & ITER & ${\pcal E}_\ell$ 	    & ITER & ${\pcal E}_\ell$   \\
\hline

1                      & 6    &  4.31e-07		 & 6	&  5.61e-07		   & 9    & 2.45e-08 \\
2                      & 9    &  3.24e-05		 & 15	&  2.26e-06		   & 9    & 1.56e-08 \\
3                      & 7    &  1.24e-06		 & 8	&  1.75e-06		   & 9    & 1.04e-07 \\
4                      & 14   &  2.21e-04		 & 8	&  2.12e-04		   & 10   & 1.74e-08 \\
5                      & 7    &  2.24e-05		 & 7	&  2.35e-05		   & 10   & 5.16e-09 \\
6                      & 10   &  4.84e-04		 & 19	&  3.72e-04		   & 10   & 1.99e-08 \\
7                      & 7    &  1.20e-06		 & 7	&  1.20e-06		   & 9    & 6.47e-08 \\
8                      & 7    &  7.11e-07		 & 7	&  7.66e-07		   & 9    & 7.68e-09 \\
9                      & 7    &  5.53e-06		 & 7	&  5.98e-06		   & 9    & 1.32e-09 \\
10                     & 6    &  6.98e-07		 & 6	&  4.92e-07		   & 8    & 8.68e-09 \\

 \hline
\end{tabular}
\end{table}
\begin{table} [h!]
\footnotesize
\centering
\caption{{\rm Network: \textsf{ITwiki}, $h(t) = \sinh(t)$ (${\rm tol} = 10^{-6}$).}}
\label{tab:itwiki2}
\begin{tabular}{c|cccccc}
\hline
                     	      &\multicolumn{2}{c}{First approach} &\multicolumn{2}{c}{Second approach}  &\multicolumn{2}{c}{Third approach}\\
                              & ITER	& ${\pcal E}_\ell$	  & ITER	 & ${\pcal E}_\ell$	& ITER & ${\pcal E}_\ell$    \\

\hline
1                             & 5  &  3.88e-08  	      & 5  &  2.90e-08  		& 6  &  8.02e-09 \\
2                             & 10 &  4.72e-05  	      & 9  &  4.68e-05  		& 7  &  1.27e-08 \\
3                             & 5  &  3.20e-08  	      & 5  &  3.17e-08  		& 6  &  7.01e-09 \\
4                             & 7  &  2.31e-05  	      & 9  &  2.33e-05  		& 8  &  4.31e-09 \\
5                             & 8  &  4.20e-05  	      & 20 &  5.77e-05  		& 8  &  5.91e-09 \\
6                             & 9  &  2.19e-04  	      & 24 &  2.13e-04  		& 8  &  2.70e-08 \\
7                             & 6  &  4.26e-07  	      & 6  &  5.85e-07  		& 7  &  3.15e-09 \\
8                             & 14 &  1.91e-04  	      & 29 &  2.24e-04  		& 8  &  3.38e-09 \\
9                             & 5  &  8.57e-08  	      & 5  &  9.31e-08  		& 6  &  5.07e-09 \\
10                            & 9  &  9.36e-06  	      & 8  &  1.12e-05  		& 8  &  3.22e-10 \\ 
 \hline
\end{tabular}
\end{table}
As one can see from the tables, only a small number of steps is 
required for all the three approaches. 
The third approach appears to be the best one for computing these quantities since it requires 
almost always the same 
number of steps for all the nodes in each network in the dataset, while attaining
higher accuracy. 
Somewhat inferior results (in terms of both the number of iterations performed and the accuracy of the computed solution)  
are obtained with the other two approaches, which however return 

\noindent very good results as well. 
This can be explained
by observing that the function $\sinh(t)$ 
being applied to the larger (approximate) singular values of $A$ takes
much larger values 
than the function $t^{-1}\sinh(t)$ used by the other two approaches,
therefore a small relative error can be attained in fewer steps (since the
largest singular values are the first to converge).

\subsection*{Resolvent-based communicability between nodes}
Our second set of numerical experiments concerns the computation of the 
resolvent-based
communicability between two nodes $i$ and $j$. 
The concerned function is now $h(t) = \frac{\alpha t}{1-(\alpha t)^2}$, where 
$\alpha\in(0,\sigma_1^{-1})$ 
is a user-defined parameter. 
The generalized matrix function $\gh(A)$ arises as the top right square block of 
the matrix resolvent $(I-\alpha\mathscr{A})^{-1}$, where the matrix 
$\mathscr{A}$ is defined as in~\eqref{eq:calA}. 
This resolvent function is similar to one first used by Katz to assign centrality indices to nodes
in a network, see \cite{Katz}.
In~\cite{BEK13} the authors showed that when $\mathscr{A}$ is as in \eqref{eq:calA},
the resolvent can be written as
$$ 
(I-\alpha\mathscr{A})^{-1} = 
\left(
\begin{array}{cc}
(I-\alpha^2 AA^T)^{-1} & \gh(A) \\
\gh(A^T) & (I-\alpha^2 A^TA)^{-1}
\end{array}
\right), \qquad \alpha\in(0,\sigma_1^{-1}).
$$ 
Furthermore, the entries of its top right block can be used 
to account for the communicability between node $i$ 
(playing the role of spreader of information, or hub) and node $j$ 
(playing the role of receiver, or authority). 
As before, the function $g(t) = (\sqrt{t})^{-1}h(\sqrt{t})$ is not completely monotonic. 
Thus the Gauss rule can only be expected to provide an approximation to the quantity of interest.

We have performed three different tests on the network \textsf{Roget} for three different values of $\alpha$. 
More in detail, we have tested $\alpha = \frac{1}{8\sigma_1},\frac{1}{2\sigma_1}$, and
$\frac{17}{20\sigma_1} = \frac{0.85}{\sigma_1}$. 
Figure~\ref{fig:sv_gKatz_fg} shows the values of the diagonal entries 
of $h(\Sigma_r)$ and $\Sigma_r^\dagger h(\Sigma_r)$ for the three different values of $\alpha$ used in the tests.  
Figure~\ref{fig:sv_alpha_fg} plots the respective
behavior of the diagonal entries of $\Sigma_r^\dagger h(\Sigma_r)$ and 
$h(\Sigma_r)$ for the three values of the parameter $\alpha$.  
From these plots, one can expect the first and second approach to require a higher number of steps than that required by the third one; 
this is because the leading singular values are mapped to appreciably larger values when
 applying the function $h(t)$ than when applying the function $t^{-1}h(t)$. 
The results for this set of experiments are contained in 
Tables~\ref{tab:Roget11}--\ref{tab:Roget33}, 
when the tolerance for the stopping criterion~\eqref{eq:sc} is set to 
${\rm tol} = 10^{-4}$. 
The pairs of nodes whose communicability we want to approximate are chosen uniformly at random among all the possible pairs of distinct nodes in the graph. 
We kept the same set of pairs in all the three experiments. 
We want to point out, however, that the value being computed for each pair varies with $\alpha$, and thus the results (in terms of number of iterations and accuracy) 
cannot be compared among the three tables. 

As can be clearly seen from the tables, the fastest and most accurate method 
(in terms of relative error with respect to the exact value) 
is once again the third. 
Indeed, it requires fewer steps than the first two approaches and it achieves a higher level of accuracy 
(see, e.g., Table~\ref{tab:Roget22}). 
In fact, in some cases the first two approaches stabilize at a value which is far from the quantity that needs to be computed; this kind of stagnation leads to the termination criterion to be
satisfied even if convergence has not been attained.
Moreover, in \tref{tab:Roget11}, case $4$ requires $\ell = {\rm rank}(A) = 992$ steps to satisfy \eqref{eq:sc} for the first two approaches, 
whereas the third only requires $5$ steps.


\begin{figure}
\centering
\includegraphics[width=\textwidth]{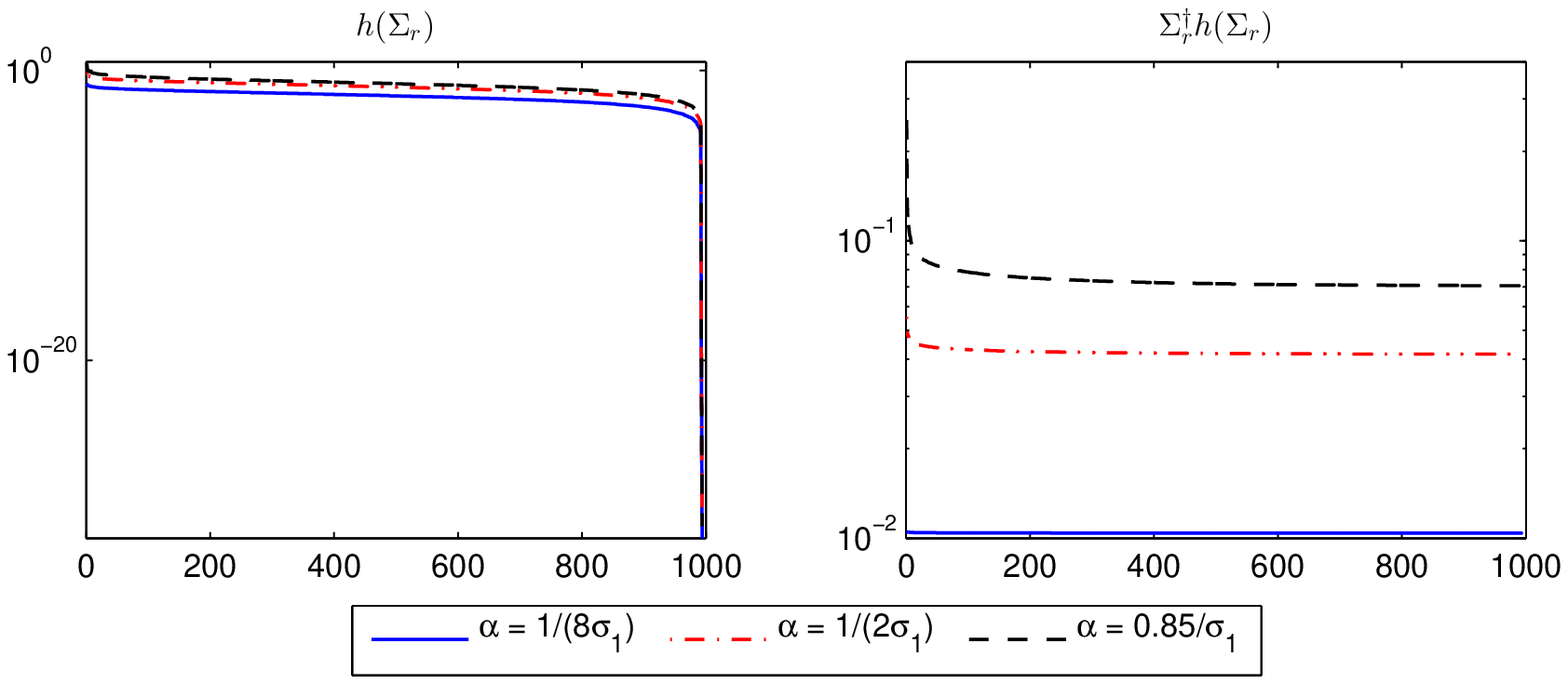}
\caption{{\rm Network: \textsf{Roget}. Diagonal entries of $h(\Sigma_r)$ 
and $\Sigma_r^\dagger h(\Sigma_r)$ for $h(t) = \frac{\alpha t}{1-(\alpha t)^2}$, 
when $\alpha = \frac{1}{8\sigma_1},\frac{1}{2\sigma_1},0.85\ \sigma_1^{-1}$}}
\label{fig:sv_gKatz_fg}
\end{figure}

\begin{figure}
\centering
\includegraphics[width=\textwidth]{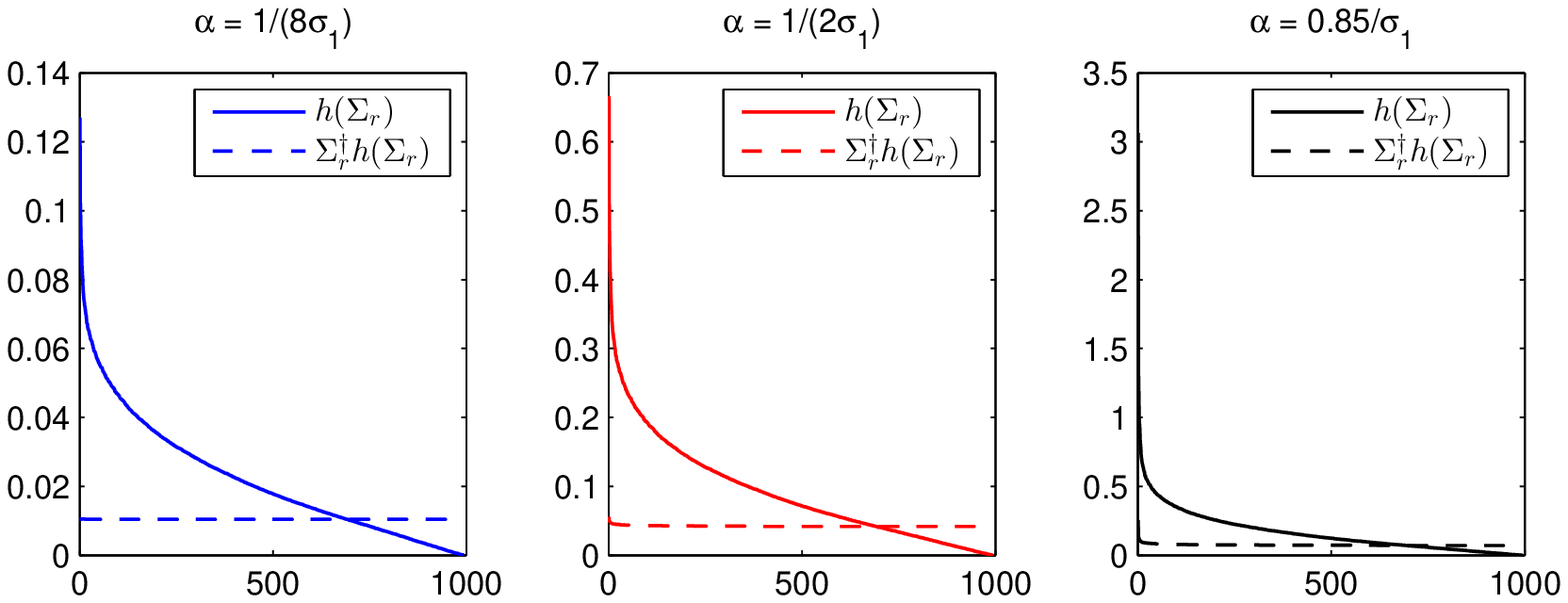}
\caption{{\rm Network: \textsf{Roget}. Diagonal entries of $h(\Sigma_r)$ 
and $\Sigma_r^\dagger h(\Sigma_r)$ for $h(t) = \frac{\alpha t}{1-(\alpha t)^2}$, 
when $\alpha = \frac{1}{8\sigma_1},\frac{1}{2\sigma_1},0.85\ \sigma_1^{-1}$}}
\label{fig:sv_alpha_fg}
\end{figure}


\begin{table}
\footnotesize
\centering
\caption{{\rm Network: \textsf{Roget}, $h(t) = \frac{\alpha t}{1-(\alpha t)^2}$, $\alpha =\frac{1}{8\sigma_1}$ (${\rm tol} = 10^{-4}$).}}
\label{tab:Roget11}
\begin{tabular}{c|cccccc}
\hline
                     	      &\multicolumn{2}{c}{First approach} &\multicolumn{2}{c}{Second approach}  &\multicolumn{2}{c}{Third approach}\\
                              & ITER	& ${\pcal E}_\ell$	  & ITER	 & ${\pcal E}_\ell$	& ITER & ${\pcal E}_\ell$    \\

\hline

1                            & 75  &  2.14e+03  & 75  &  2.14e+03 & 5  &  6.61e-08 \\
2                            & 106 &  1.60e-02  & 106 &  1.60e-02 & 4  &  3.75e-08 \\
3                            & 906 &  2.99e-01  & 906 &  2.99e-01 & 5  &  1.65e-08 \\
4                            & 992 &  2.17e-08  & 992 &  7.82e-06 & 5  &  1.33e-07 \\
5                            & 166 &  7.16e+02  & 166 &  7.16e+02 & 5  &  7.52e-08 \\
6                            & 257 &  1.03e-01  & 257 &  1.03e-01 & 4  &  4.46e-08 \\
7                            & 874 &  5.41e-01  & 874 &  5.41e-01 & 5  &  2.46e-08 \\
8                            & 274 &  1.06e+00  & 274 &  1.06e+00 & 5  &  9.46e-08 \\
9                            & 259 &  8.37e-04  & 259 &  8.37e-04 & 5  &  7.31e-11 \\
10			     & 733 &  2.48e+01  & 733 &  2.48e+01 & 5  &  3.37e-07 \\

\hline

\end{tabular}
\end{table}

\begin{table}
\footnotesize
\centering
\caption{{\rm Network: \textsf{Roget}, $h(t) = \frac{\alpha t}{1-(\alpha t)^2}$, $\alpha = \frac{1}{2\sigma_1}$ (${\rm tol} = 10^{-4}$).}}
\label{tab:Roget22}
\begin{tabular}{c|cccccc}
\hline
                     	      &\multicolumn{2}{c}{First approach} &\multicolumn{2}{c}{Second approach}  &\multicolumn{2}{c}{Third approach}\\
                              & ITER	& ${\pcal E}_\ell$	  & ITER	 & ${\pcal E}_\ell$	& ITER & ${\pcal E}_\ell$    \\

\hline
1                            & 75  &  6.32e+00 & 75  &  6.32e+00 & 7  &  1.90e-06 \\
2                            & 87  &  8.90e-03 & 87  &  8.90e-03 & 6  &  6.02e-07 \\
3                            & 308 &  1.10e-02 & 308 &  1.10e-02 & 6  &  7.96e-06 \\
4                            & 106 &  5.71e-01 & 106 &  5.72e-01 & 7  &  7.77e-08 \\
5                            & 495 &  1.72e-01 & 495 &  1.72e-01 & 7  &  4.43e-06 \\
6                            & 79  &  4.51e-02 & 79  &  4.51e-02 & 6  &  2.12e-06 \\
7                            & 118 &  8.64e-02 & 118 &  8.64e-02 & 7  &  4.24e-07 \\
8                            & 121 &  1.00e-01 & 121 &  1.01e-01 & 7  &  5.84e-07 \\
9                            & 59  &  1.91e-02 & 59  &  1.91e-02 & 6  &  6.99e-07 \\
10			     & 574 &  1.87e-01 & 574 &  1.87e-01 & 7  &  1.49e-06 \\

\hline

\end{tabular}
\end{table}

\begin{table}
\footnotesize
\centering
\caption{{\rm Network: \textsf{Roget}, $h(t) = \frac{\alpha t}{1-(\alpha t)^2}$, $\alpha = 0.85\ \sigma_1^{-1}$ (${\rm tol} = 10^{-4}$).}}
\label{tab:Roget33}
\begin{tabular}{c|cccccc}
\hline
                     	      &\multicolumn{2}{c}{First approach} &\multicolumn{2}{c}{Second approach}  &\multicolumn{2}{c}{Third approach}\\
                              & ITER	& ${\pcal E}_\ell$	  & ITER	 & ${\pcal E}_\ell$	& ITER & ${\pcal E}_\ell$    \\

\hline
1                            & 74  &  2.72e-01  & 74  &  2.72e-01 & 10  &  4.88e-06 \\
2                            & 66  &  3.45e-03  & 66  &  3.45e-03 & 8	&  5.22e-06 \\
3                            & 97  &  6.96e-02  & 97  &  6.96e-02 & 8	&  4.79e-06 \\
4                            & 58  &  3.72e-02  & 58  &  3.72e-02 & 8	&  8.04e-05 \\
5                            & 147 &  6.23e-02  & 147 &  6.23e-02 & 9	&  1.23e-05 \\
6                            & 53  &  8.48e-03  & 53  &  8.48e-03 & 9	&  3.34e-06 \\
7                            & 74  &  1.58e-02  & 74  &  1.58e-02 & 7	&  3.20e-04 \\
8                            & 117 &  6.49e-03  & 117 &  6.49e-03 & 10  &  4.52e-06 \\
9                            & 23  &  6.02e-03  & 23  &  6.02e-03 & 9	&  4.34e-07 \\
10			     & 152 &  1.70e-01  & 152 &  1.70e-01 & 9	&  3.90e-05 \\
\hline

\end{tabular}
\end{table}

\subsection*{Comparison with standard Lanczos-based approach}

In the case of generalized matrix functions like $\sinh^\diamond (A)$, which
occur as submatrices of ``standard" matrix functions 
applied to the symmetric matrix $\mathscr{A}$, it is natural to 
compare the previously proposed approaches with the use of Gauss quadrature-based
bounds and estimates based on the Lanczos process. This was the approach used, for
example, in \cite{BEK13}. Henceforth, we refer to this approach as ``the {\tt mmq} approach," 
since it is implemented on the basis of the ${\tt mmq}$ toolkit \cite{mmq} originally developed by 
G\'erard Meurant; see also \cite{BookGM}.

Numerical experiments, not shown here, indicate that on average, the 
{\tt mmq} approach requires a slightly higher number of iterations
than our third approach to deliver comparable accuracy in computing 
the communicability between pairs of nodes. Note that the cost per
step is comparable for the two methods. An advantage of the
{\tt mmq} approach is that it can provide lower and upper bounds on the 
quantities being computed, but only if
bounds on the singular values of $A$ are available. 
A disadvantage is that it requires working with vectors of length $2n$
instead of $n$.

Of course, the  Lanczos-based approach is not applicable to 
generalized matrix functions that do not arise as submatrices of 
standard matrix functions.
 
\subsection*{Block approaches}
In the following, we test the performance of the two block approaches described 
in section~\ref{sec:approach_bl} 
when trying to approximate the communicabilities among $k$ nodes of the \textsf{Twitter} network, which has $3656$ 
nodes and $188712$ edges~\cite{websnap}. 
All the computations were carried out with MATLAB Version 7.10.0.499 (R2010a) 64-bit 
for Linux, in double precision arithmetic, on an Intel Core i5 computer with 4 GB RAM.

In order to compute approximations to the communicabilities, we set 
$$
Z = W = [\mathbf{e}_{i_1}, {\bf e}_{i_2}, \dots, \mathbf{e}_{i_k}],
$$ 
where $i_1,i_2, \dots, i_k$ are chosen uniformly at random among the $n$ nodes of the network. 
To better analyze the behavior of the methods, we run both algorithms ten times and report in each table the averaged values obtained by changing the set of $k$ nodes after each run. 
As in the previous subsection, we test the performance of the two generalized matrix functions induced by $h(t) = \sinh(t)$ and 
$h(t) = \alpha t (1-(\alpha t)^2)^{-1}$, respectively.

The first approach is based on the computation of block Gauss and anti-Gauss quadrature rules. Since $\widetilde{Z} \neq W$, we need to use the nonsymmetric Lanczos algorithm and, in order to avoid breakdown during the computation, it is convenient to add a dense vector to each initial block (see~\cite{BDY} for more details). Each table reports the relative error and the relative distance between the two quadrature rules computed as:

$$
{\pcal E}_\ell = \frac{\|F_\ell - Z^T \gh(A) W \|_2}{\|Z^T \gh(A) W  \|_2}
\qquad \text{and} \qquad
{\pcal R}_\ell = 
\frac{\|{\pcal G}_\ell - {\pcal H}_{\ell+1}\|_{\max}}
{\|{\pcal G}_\ell + {\pcal H}_{\ell+1}\|_{\max}} \, ,
$$
respectively, where  
$$\|M\|_{\max} = \smash{\displaystyle\max_{\substack{1\le i\le m \\ 1\le j\le n}}} \{M_{ij} \} \text{, with } M\in\Cmn,$$ 
and $F_\ell$ is given by~\eqref{FL}.

In order to obtain good entrywise approximations of $Z^T \gh(A) W$, 
the domain $D$ of analyticity of $g(t) = (\sqrt{t})^{-1} h(\sqrt{t})$ 
has to enclose the smallest interval containing the spectrum 
of $A^TA$, and the boundary $\partial D$ has to be well separated 
from the extremes of the interval.  
However, when $h(t) = \sinh(t)$, the function $g(t)$ 
applied to our test problems
does not exhibit this nice property; indeed, in all three
cases $A$ is (numerically) singular and therefore
the singularity $t = 0$ belongs to the smallest interval containing the spectrum $A^TA$. 
As pointed out in section~\ref{sec:approach_bl}, we cannot therefore
expect $F_\ell = \left({\pcal G}_\ell + {\pcal H}_{\ell+1}\right)/2$
to provide a good approximation for $Z^T \gh(A) W$, 
since the condition $\left({\pcal I}-{\pcal H}_{\ell+1}\right)g \approx -\left({\pcal I}-{\pcal G}_\ell\right)g$ is not guaranteed.  

The results in Table~\ref{tab:block2} confirm this observation, 
since they 
clearly show that the relative error ${\pcal E}_{\ell}$ is not comparable with the relative distance ${\pcal R}_{\ell}$. 
As expected, a small value of ${\pcal R}_{\ell}$ does not ensure a 
satisfactory value of ${\pcal E}_{\ell}$. 
Therefore, the relative distance between the approximations provided by 
the Gauss and anti-Gauss rules cannot be used as a stopping criterion. 
Moreover, the results in Table~\ref{tab:block2} also show that performing more 
iterations does not improve the results; indeed, 
for all the values of the block size $k$ it holds ${\pcal E}_{5} \approx {\pcal E}_{10}$. 
It is also worth noting that the relative distance does not decrease 
as $\ell$ increases, but stabilizes 
far from the desired value and in one case it even increases. 
In view of this, the behavior of the algorithm is not satisfactory regardless 
of the nodes taken into account or the block size $k$.

\begin{table}[t]
\footnotesize
\caption{{\rm Execution time (in seconds), relative error and relative 
distance for the computation of the total communicabilities between $k$ nodes of the \textsf{Twitter} network with $\ell = 5$ and $\ell =10$ steps.}} 
\label{tab:block2}
\begin{center}
\begin{tabular}{c|ccc|ccc}
\hline
\multirow{2}{*}{$k$} & \multicolumn{3}{c|}{$\ell = 5$} & \multicolumn{3}{c}{$\ell = 10$} \\
 & Time & ${\pcal E}_5$ & ${\pcal R}_5$ & Time & ${\pcal E}_{10}$ & ${\pcal R}_{10}$ \\ 
\hline
5   & 2.14e-01 & 4.62e-04 & 5.07e-09 & 3.50e-01 & 4.62e-04 & 9.74e-10\\
10  & 2.70e-01 & 1.04e-02 & 2.21e-09 & 5.62e-01 & 1.04e-02 & 9.96e-10\\ 
20  & 4.21e-01 & 3.78e-02 & 5.39e-10 & 1.10e+00 & 3.78e-02 & 8.12e-09 \\
30  & 6.63e-01 & 2.24e-02 & 1.78e-11 & 2.12e+00 & 2.24e-02 & 3.14e-10 \\
50  & 1.24e+00 & 4.59e-02 & 6.83e-12 & 5.57e+00 & 4.59e-02 & 1.63e-11 \\
100 & 3.86e+00 & 5.65e-02 & 3.43e-11 & 2.72e+01 & 5.65e-02 & 1.60e-11 \\
\hline
\end{tabular}
\end{center}
\end{table}


Tables~\ref{tab:block3}--\ref{tab:block5} report the performance of the method when trying to approximate the communicabilities with respect to the function $h(t) = \alpha t (1 - (\alpha t)^2)^{-1}$, using different values of $\alpha \in (0, \sigma_1^{-1})$.
In this case, the function $g(t) = \alpha(1 - \alpha^2 t)^{-1}$ is analytic, provided that $\alpha < \sigma_1^{-1}$, as in our case. 

Note that as $\alpha$ approaches $\sigma_1^{-1}$, the relative error ${\pcal E}_\ell$ increases 
(cf. Tables~\ref{tab:block3}--\ref{tab:block5}). 
This happens because the distance between the boundary $\partial D$ of the domain $D$ of analyticity of $g(t)$ 
and the smallest interval containing the spectrum of $A^TA$ is decreasing (see section~\ref{ssec:block1}). 
In the first table, the values of ${\pcal E}_{2}$ and ${\pcal R}_2$ are comparable and we reach the desired accuracy with only two iterations. 
We remark that some other experiments, not shown here, pointed out that increasing the number of steps of the algorithm does not improve the relative error. 

\begin{table}[t]
\footnotesize
\caption{{\rm Number of block Lanczos steps, execution time (in seconds), 
relative error and relative distance for the computation of the 
Katz communicabilities between $k$ nodes of the \textsf{Twitter} network. 
Here, $\alpha = 1/(8\sigma_1)$.}}
\label{tab:block3}
\begin{center}
\begin{tabular}{c|ccc}
\hline
$k$ & Time & ${\pcal E}_2$ & ${\pcal R}_2$ \\ 
\hline
5   & 5.20e-02 & 4.77e-04 & 1.99e-04 \\
10  & 8.89e-02 & 7.90e-05 & 2.85e-04 \\ 
20  & 1.69e-01 & 4.06e-04 & 2.59e-04 \\
30  & 2.71e-01 & 5.85e-04 & 1.50e-04 \\
50  & 1.31e+00 & 2.60e-04 & 1.53e-04 \\
100 & 6.75e-01 & 5.40e-04 & 1.83e-04 \\
\hline
\end{tabular}
\end{center}
\end{table}

\begin{table}[h!]
\footnotesize
\caption{{\rm Number of block Lanczos steps, execution time (in seconds), 
relative error and relative distance for the computation of the Katz
communicabilities between $k$ nodes of the \textsf{Twitter} network. 
Here, $\alpha = 1/(2\sigma_1)$.}} 
\label{tab:block4}
\begin{center}
\begin{tabular}{c|cc c}
\hline
$k$ &  Time & ${\pcal E}_3$ & ${\pcal R}_3$ \\ 
\hline
5   & 5.63e-02 & 6.20e-04 & 7.54e-05 \\
10  & 9.47e-02 & 4.50e-04 & 2.66e-05 \\ 
20  & 1.79e-01 & 3.36e-03 & 1.15e-05 \\
30  & 2.61e-01 & 2.41e-03 & 6.56e-07 \\
50  & 4.66e-01 & 8.36e-03 & 1.14e-06 \\
100 & 1.35e+00 & 9.21e-03 & 1.04e-07 \\
\hline
\end{tabular}
\end{center}
\end{table}

\begin{table}[h!]
\footnotesize
\caption{{\rm Number of block Lanczos steps, execution time (in seconds), 
relative error and relative distance for the computation of the communicabilities 
between $k$ nodes of the \textsf{Twitter} network. Here, $\alpha = 0.85/\sigma_1$.}} 
\label{tab:block5}
\begin{center}
\begin{tabular}{c|c cc}
\hline
$k$  & Time & ${\pcal E}_3$ & ${\pcal R}_3$ \\ 
\hline
5   & 8.07e-02 & 1.89e-03 & 5.04e-05 \\
10  & 9.74e-02 & 4.26e-03 & 2.76e-04 \\ 
20  & 1.83e-01 & 2.19e-02 & 3.25e-04 \\
30  & 2.95e-01 & 1.45e-02 & 6.48e-05 \\
50  & 5.27e-01 & 3.46e-02 & 7.41e-05 \\
100 & 1.33e+00 & 1.90e-02 & 7.27e-06 \\
\hline
\end{tabular}
\end{center}
\end{table}

Table~\ref{tab:block4} shows that we obtain good approximations setting 
$\alpha = 0.5 \sigma_1^{-1}$, 
but the relative distance between the two quadrature rules decreases faster than the relative error. 
A similar behavior is shown in Table~\ref{tab:block5}, where $\alpha = 0.85\sigma_1^{-1}$. The relative error increases as $\alpha$ gets closer to $\sigma_1^{-1}$ and does not improve performing more steps.

\begin{table}[t]
\footnotesize
\caption{{\rm Execution time (in seconds), relative error, number of block GK steps and stopping criterion for the computation of the communicabilities between $k$ nodes of the \textsf{Twitter} network.}} 
\label{tab:block6}
\begin{center}
\begin{tabular}{c|cccc}
\hline
$k$ & ITER  & Time & ${\pcal E}_\ell$ & ${\pcal R}_\ell$ \\ 
\hline
5   & 6     & 2.51e+00 & 1.14e-08     & 1.98e-06 \\
10  & 6     & 1.65e+00 & 5.79e-09     & 1.58e-06 \\ 
20  & 5     & 1.77e+00 & 5.21e-09     & 1.20e-06 \\
30  & 5     & 2.19e+00 & 5.05e-09     & 1.25e-06 \\
50  & 5     & 3.34e+00 & 1.84e-09     & 7.04e-07 \\
100 & 4     & 5.05e+00 & 6.65e-09     & 4.38e-06 \\
\hline
\end{tabular}
\end{center}
\end{table}

We turn now to the approximation of the quantity~\eqref{gf_bilinear} using the 
second block approach, namely the block Golub--Kahan decomposition algorithm.
When using this approach, we perform as many steps as necessary to obtain   
$$
{\pcal R}_\ell = \frac{\|F_\ell - F_{\ell-1} \|_2}{\| F_{\ell-1}\|_2} < {\rm tol},
$$
with ${\rm tol} = 10^{-5}$. 
Table~\ref{tab:block6} displays the results obtained when we approximate the communicabilities among 
$k$ nodes with respect to the generalized matrix function induced 
by $h(t) = \sinh(t)$. 
Clearly, this approach requires a small number of steps to reach a high level of accuracy. 

\begin{table}[t!]
\footnotesize
\caption{{\rm Execution time (in seconds), relative error, number of block Golub--Kahan
 steps and stopping criterion for the computation of the communicabilities among $k$ 
nodes of the \textsf{Twitter} network. In this case, $\alpha = 1/(8\sigma_1)$.}} 
\label{tab:block7}
\begin{center}
\begin{tabular}{c|cccc}
\hline
$k$ & ITER & Time     & ${\pcal E}_\ell$ & ${\pcal R}_\ell$ \\ 
\hline
5   & 3    & 7.62e-01 & 3.22e-10    	 & 2.14e-06 \\
10  & 3    & 8.01e-01 & 3.93e-10    	 & 2.16e-06 \\ 
20  & 3    & 9.27e-01 & 1.77e-10    	 & 2.24e-06 \\
30  & 3    & 1.09e+00 & 9.64e-11    	 & 1.74e-06 \\
50  & 3    & 1.49e+00 & 4.40e-11    	 & 6.70e-07 \\
100 & 3    & 2.86e+00 & 1.16e-11    	 & 4.21e-07 \\
\hline
\end{tabular}
\end{center}
\end{table}

\begin{table}[h!]
\footnotesize
\caption{{\rm Execution time (in seconds), relative error, number of block GK steps and stopping criterion for the computation of the communicabilities between $k$ nodes of the \textsf{Twitter} network. In this case, $\alpha = 1/(2\sigma_1)$.}} 
\label{tab:block8}
\begin{center}
\begin{tabular}{c|cccc}
\hline
$k$ & ITER  & Time     & ${\pcal E}_\ell$ & ${\pcal R}_\ell$ \\ 
\hline
5   & 4     & 9.35e-01 & 3.47e-08    	  & 2.05e-06 \\
10  & 4     & 1.08e+00 & 8.85e-09    	  & 1.90e-06 \\ 
20  & 4     & 1.38e+00 & 1.56e-09    	  & 1.20e-06 \\
30  & 4     & 1.65e+00 & 3.70e-10    	  & 3.59e-07 \\
50  & 4     & 2.32e+00 & 1.49e-10    	  & 2.56e-07 \\
100 & 4     & 4.91e+00 & 2.20e-11    	  & 6.21e-08 \\
\hline
\end{tabular}
\end{center}
\end{table}

\begin{table}[h!]
\footnotesize
\caption{{\rm Execution time (in seconds), relative error, number of block GK steps and stopping criterion for the computation of the communicabilities between $k$ nodes of the \textsf{Twitter} network. In this case, $\alpha = 0.85/\sigma_1$.}} 
\label{tab:block9}
\begin{center}
\begin{tabular}{c|cccc}
\hline
$k$ & ITER & Time     & ${\pcal E}_\ell$ & ${\pcal R}_\ell$ \\ 
\hline
5   & 5    & 1.29e+00 & 3.02e-08	 & 2.08e-06 \\
10  & 5    & 1.36e+00 & 1.01e-08	 & 6.86e-07 \\ 
20  & 5    & 1.66e+00 & 1.42e-08	 & 2.38e-06 \\
30  & 5    & 2.16e+00 & 3.49e-09	 & 1.33e-06 \\
50  & 4    & 2.58e+00 & 7.81e-09	 & 4.33e-06 \\
100 & 4    & 4.84e+00 & 1.73e-09	 & 1.62e-06 \\
\hline
\end{tabular}
\end{center}
\end{table}

Tables~\ref{tab:block7}-\ref{tab:block9} show the results concerning the approximation 
of the communicabilities among $k$ nodes using the generalized matrix function 
induced by $h(t) = \alpha t(1-(\alpha t)^2)^{-1}$. 
As above, we consider three different values for the parameter $\alpha$. 
As in the scalar case, the method requires fewer iterations to reach a higher 
accuracy as the value of $\alpha$ moves away from $\sigma_1^{-1}$.

The two approaches behave again very differently. As before, the
results obtained with the second approach are very promising also in view of the 
fact that we did not make any assumptions on the 
regularity of the function.

\section{Conclusions}
\label{sec:conclusions}

In this paper we have proposed several algorithms for the computation of
certain quantities associated with generalized matrix functions. 
These techniques are based on Gaussian quadrature rules and different
variants of the Lanczos and Golub--Kahan algorithms.
In particular, we have investigated three distinct approaches for estimating
scalar quantities like ${\bf z^T}\gf (A) {\bf w}$,
and two block methods for computing
matrix-valued expressions like $Z^T \gf (A) W$. The performance 
of the various approaches has been tested in the context of computations
arising in network theory. 
While not all methods can be expected to always perform well in practice,
we have identified two approaches (one scalar-based, the other
block-based) that produce fast and accurate approximations for the
type of problems considered in this paper.


\section*{Acknowledgments}
Francesca Arrigo and Caterina Fenu would like to thank the Department
of Mathematics and Computer Science of Emory University for the hospitality 
offered in 2015, when part of this work was completed. 


\end{document}